\documentclass[12pt,a4paper]{article}
\usepackage[cp1251]{inputenc}
\usepackage{amsmath,amsthm}
\usepackage{amsfonts,amstext,amssymb,verbatim,epsfig}
\usepackage{dsfont}
\usepackage{enumerate}
\usepackage{psfrag}
\usepackage{graphicx}
\usepackage{subfig}
\usepackage{color}
\usepackage{float}
\usepackage{subfig}
\usepackage{enumitem}
\usepackage{mathrsfs}
\usepackage{hyperref}
\usepackage{xcolor}
\hypersetup{
    colorlinks,
    linkcolor={blue!50!black},
    citecolor={blue!50!black},
    urlcolor={blue!80!black}
}

\usepackage[top=3cm, bottom=3.5cm, left=2.5cm, right=2.5cm]{geometry} 

\def\R{{\mathbb{R}}}
\def\N{{\mathbb{N}}}
\def\Z{{\mathbb{Z}}}
\def\P{{\mathbb{P}}}
\def\E{{\mathbb{E}}}

\newtheorem{theorem}{Theorem}[section]
\newtheorem{corollary}[theorem]{Corollary}
\newtheorem{lemma}[theorem]{Lemma}
\newtheorem{proposition}[theorem]{Proposition}

\newtheoremstyle{likedef}
  {}%
  {}%
  {}%
  {}
  {\bfseries}%
  {.}%
  {.5em}%
  {}%

\theoremstyle{likedef}

\newtheorem{definition}[theorem]{Definition}

\numberwithin{equation}{section}

\date{}

\begin{document}

\title{On questions of uniqueness for the vacant set of Wiener sausages and Brownian interlacements}
\author{
Yingxin Mu\thanks{
University of Leipzig, Department of Mathematics,  
Augustusplatz 10, 04109 Leipzig, Germany.
Email: yingxin.mu@uni-leipzig.de and artem.sapozhnikov@math.uni-leipzig.de. 
The research of both authors has been supported by the DFG Priority Program 2265 ``Random Geometric Systems'' (Project number 443849139).}
\and
Artem Sapozhnikov\footnotemark[1]
}

\maketitle


\begin{abstract}
We consider connectivity properties of the vacant set of (random) ensembles of Wiener sausages in $\R^d$ in the transient dimensions $d \geq 3$. We prove that the vacant set of Brownian interlacements contains at most one infinite connected component almost surely. For finite ensembles of Wiener sausages, we provide sharp polynomial bounds on the probability that their vacant set contains at least $2$ connected components in microscopic balls. The main proof ingredient is a sharp polynomial bound on the probability that several Brownian motions visit jointly all hemiballs of the unit ball while avoiding a slightly smaller ball.
\end{abstract}

\section{Introduction}
In this paper, we are interested in connectivity properties of the complement of (random) ensembles of Wiener sausages in $\R^d$ in the transient dimensions $d\geq 3$. Our main motivation comes from the question about uniqueness of the infinite connected component in the vacant set of Brownian interlacements in $\R^d$. Brownian interlacements, introduced by Sznitman in \cite{Sznitman-BI}, is the continuous counterpart of random interlacements on $\Z^d$ (see \cite{Sznitman-AM}); 
while random interlacements is a Poisson cloud of doubly-infinite random walks on $\Z^d$, Brownian interlacements is a Poisson cloud of doubly-infinite Wiener sausages of fixed positive radius, whose density is controlled by a parameter $\alpha>0$. 

\smallskip

Since its introduction in $2007$, the vacant set of random interlacements has been an important example of percolation model with strong, algebraically decaying correlations. The study of its phase transition has originated several groundbreaking new perspectives on percolation models---most notably, the decoupling inequalities, see e.g.\ \cite{Sznitman-DI,PT-SLT}---, which stimulated remarkable developments in the understanding of strongly correlated lattice models (see e.g.\ \cite{CT-Book, DRS-Book, DGRS-GFF} for comprehensive literature review). 

\smallskip

Our understanding of percolation models in $\R^d$ is still limited. The attention here has been mainly on the study of the Boolean model, the subset of $\R^d$ covered by balls of random radii centered about points of a stationary Poisson point process in $\R^d$, see e.g.\ \cite{MR-Book, Gouere08} and more recent \cite{DRT-Boolean}.
Lately, there has also been interest in the vacant set of the Boolean model, see e.g.\ \cite{ATT18,Pen18}. 
Although the analysis of the Boolean model can be quite delicate---especially when the radii distribution has algebraic tail---, the main challenge generally does not come from the continuum nature of the model, since the boundaries of connected components are quite simple.---To illustrate this point, we recall an argument from the proof of the uniqueness of the infinite connected component in the vacant set of the Boolean model by Meester and Roy (see \cite[Proposition~5.4]{MR-Uniqueness}): Each connected component in the vacant set contains on its boundary at least one point of the intersection of $d$ balls of the Boolean model; thus, the local number of connected components in the vacant set is a well-behaved random variable.---One can imagine that in continuum percolation models with richer microscopic structure, 
some features not present in discrete 
models, such as large number of components on microscales, the existence of infinite connected components of finite volume or unbounded components with no path to infinity (``star-like'' components), can become the prior challenge to rebut. 
The vacant set of Brownian interlacements is one such model. 

\smallskip

Due to similarities in the constructions of Poissonian clouds as well as large-scale properties of random walks and Wiener sausages, Brownian interlacements share basic properties with random interlacements, such as slow algebraic decay of correlations and absence of the so-called finite energy property; furthermore, as shown by Li \cite{Li-BI}, Brownian interlacements is almost surely connected for any intensity parameter $\alpha>0$ and its vacant set---the complement of Brownian interlacements in $\R^d$---undergoes a non-trivial percolation phase transition in $\alpha$. As it is relevant to our paper, let us be more precise about the latter result. Li actually proves that, for large densities $\alpha$, the vacant set of Brownian interlacements consists only of bounded components almost surely and, for small positive densities $\alpha$, it contains an unbounded connected component almost surely. 
As we remarked just above, unlike in discrete models, in continuum models it is generally not immediate that an unbounded connected component contains a continuous path to infinity. Jumping ahead though, we note that for the vacant set of Brownian interlacements it is indeed the case, which follows from our results. 

\smallskip

Although our initial motivation comes from the question about uniqueness of the infinite connected component in the vacant set of Brownian interlacements, the most novel part of our work is about the analysis of the microstructure of the vacant set of finite ensembles of independent Wiener sausages and is of independent interest. 
To put these results into context, we begin by describing the local picture of the Brownian interlacements in $\R^d$ ($d\geq 3$) and postpone its full construction as a Poisson point process to Section~\ref{sec:BI}. 

\smallskip

Let $\alpha>0$ and $r>0$. The Brownian interlacements at level $\alpha$ with radius $r$ is a random closed subset of $\R^d$, whose restriction to Euclidean ball $B(0,R)$ can be sampled as follows:
\begin{itemize}
\item
let $R'\geq R+r$;
\item
let $N$ be a Poisson distributed random variable with parameter $\alpha\mathrm{cap}(B(0,R'))$, where $\mathrm{cap}(\cdot)$ is the Wiener capacity\footnote{$\mathrm{cap}(B(0,s)) = \big(2\pi^{d/2}/\Gamma(\tfrac{d-2}{2})\big)s^{d-2}$, see e.g.\ \cite[(3.55)]{Sznitman-Book}};
\item
let $W^{(1)},W^{(2)},\ldots$ be independent Brownian motions in $\R^d$, independent from $N$, started from uniform points on the boundary of $B(0,R')$. 
\end{itemize}
Then (for any $R'\geq R+r$), the union of Wiener sausages of radius $r$ around the Brownian motions $W^{(1)},\ldots, W^{(N)}$, restricted to the ball $B(0,R)$, 
\[
\Big(\bigcup\limits_{n=1}^N\bigcup\limits_{t_n=0}^\infty B\big(W^{(n)}_{t_n},r\big)\Big)\cap B(0,R),
\]
has the same law as the Brownian interlacements at level $\alpha$ with radius $r$ in $B(0,R)$. (In essence, the Brownian interlacements in the infinite volume is obtained by taking the limits of $R'$ and $R$ to infinity.) Brownian interlacements with parameters $(\alpha,r)$, uniformly scaled by a factor $\lambda>0$, has the same law as Brownian interlacements with parameters $(\lambda^{2-d}\alpha,\lambda r)$, thus in the study of scale-invariant properties of Brownian interlacements for fixed $\alpha$ and $r$, it is not a loss of generality to consider $r=1$.

\medskip

We now describe our results. 
In Theorem~\ref{thm:uniqueness-BI} we show that for any $\alpha$ and $r$, 
\begin{equation}\label{intro:uniqueness-BI}
\begin{array}{c}
\text{the vacant set of Brownian interlacements either contains no infinite}\\
\text{components almost surely or exactly one infinite component almost surely.} 
\end{array}
\end{equation}
With by now a standard argument of van den Berg and Keane \cite{BK-Continuity}, one immediately infers from \eqref{intro:uniqueness-BI} that the percolation function---the probability that the connected component of the origin in the vacant set is infinite---is continuous in $\alpha$ throughout the supercritical phase of the vacant set of Brownian interlacements (see e.g.\ \cite[Corollary~1.2]{Teixeira-Uniqueness} for the argument in the case of the vacant set of random interlacements).

\smallskip

Brownian interlacements is shift-invariant and ergodic (cf.\ \eqref{eq:BI-ergodicity}), so the number of infinite connected components is a priori constant almost surely. 
The standard approach to uniqueness in percolation models on $\Z^d$ is to rule out separately by contradiction the two cases: (a) the number of infinite components is $k$ for some $2\leq k<\infty$ and (b) the number of infinite components is infinite. In the first case, one considers a large ball, which intersects all $k$ infinite components with positive probability, and modifies the configuration locally in order to merge all of them, thus obtaining, in contradiction with initial assumption, that a unique infinite component has to exist with positive probability. In the second case, the contradiction is commonly obtained via the Burton-Keane argument \cite{BK-Uniqueness}. If the number of infinite components is infinite, then there must be a positive density of so-called trifurcations---locations where an infinite component locally splits into at least $3$ infinite branches---, which infers that infinite components are macroscopically tree-like and cannot be embedded in $\Z^d$.
Our proof will follow the same steps; however, two issues have to be dealt with. On the one hand, the finite energy property---commonly used to justify local modifications---does not hold for Brownian interlacements. Indeed, doubly infinite Brownian paths that visit a predefined box have to be rerouted locally to ensure merging of infinite components resp.\ creation of a trifurcation. On the other hand, a tree-like structure of infinite components alone does not lead to a contradiction in continuum; one still has to rule out a possibility of infinite components of finite volume. In both cases, the major obstruction is the complicated structure of the vacant set on microscales. 

We resolve the first issue by a novel rerouting strategy. We make use of the connectedness of the Brownian interlacements to reroute Brownian paths that visit a predefined box in a neighborhood outside of the box, so that most of the volume of the respective rerouted Wiener sausages is supported in the occuped set of Brownian interlacements outside of the box. This rerouting method is quite robust and saves us from dealing with the microstructure of the vacant set in this step. In order to implement it, we require a slightly stronger connectivity property of Brownian interlacements than just its connectedness, see Proposition~\ref{prop:BI-connectivity-int}. 
Before we discuss the next issue, let us mention, that our rerouting method is very different from the one used by Teixeira in \cite{Teixeira-Uniqueness} to prove the uniqueness of the infinite component for the vacant set of random interlacements on $\Z^d$. Teixeira explores the geometry of $\Z^d$ to show that one can suitably reroute random walks inside the box; we were not able to adapt these ideas in continuum. Our method is more robust and---applied in discrete setting---allows to prove uniqueness of the infinite component in the vacant set of random interlacement on general vertex-transitive amenable transient graphs, see \cite{MS-RI-amenable}.

To resolve the issue of large thin components in the vacant set, we prove in Theorem~\ref{thm:expected-number-components} that for any $R'\geq R+7(d+1)$, 
\begin{equation}\label{intro:expected-number-components}
\begin{array}{c}
\text{expected number of vacant components in annulus $\{x\in\R^d\,:\,R\leq \|x\|\leq R'\}$,}\\ 
\text{which intersect its both the inner and the outer boundaries is finite.}
\end{array}
\end{equation}
Not only \eqref{intro:expected-number-components} allows to complete the proof of uniqueness, it also immediately implies that every unbounded component contains a path to infinity. (We actually have to use this observation also in the proof of the existence of trifurcations, see Lemma~\ref{l:trifurcation}.) We do not have a direct proof of \eqref{intro:expected-number-components}; instead, we use the fact, that the large number of connected components in continuum infers the existence of multiple components on microscales, and reduce \eqref{intro:expected-number-components} to a question about---mind the local picture of Brownian interlacements---the microstructure of the vacant set of finite ensembles of Wiener sausages. 

\smallskip

Let $W^{(1)},\ldots, W^{(K)}$ be independent Brownian motions in $\R^d$ ($d\geq 3$) started on the boundary of ball $B(0,2)$ and let 
\[
\mathcal V_K = \R^d\setminus\Big(\bigcup\limits_{k=1}^K\bigcup\limits_{t_k=0}^\infty B\big(W^{(k)}_{t_k},1\big)\Big)
\]
be the vacant set of the ensemble of $K$ respective Wiener sausages of radius $1$. 
In Theorem~\ref{thm:uniqueness-severalballs} we prove that for any $K$ and $\varepsilon\in(0,1)$, 
\begin{equation}\label{intro:uniqueness-severalballs}
\mathsf P\left[
\begin{array}{c}
\text{$\mathcal V_K\cap B(0,\varepsilon)$ contains at least $2$}\\
\text{connected components}
\end{array}
\right] \leq C\log^m\big(\tfrac1\varepsilon\big)\varepsilon^{d+1},
\end{equation}
for some dimension dependent constants $C$ and $m$. (In fact,  Theorem~\ref{thm:uniqueness-severalballs} estimates more generally the probability of simultaneous microscopic nonuniqueness in several well separated $\varepsilon$-balls.) By the local picture of Brownian interlacements, the bound \eqref{intro:uniqueness-severalballs} holds immediately also for the vacant set of Brownian interlacements at any level $\alpha$. 

\smallskip

The event in \eqref{intro:uniqueness-severalballs} exerts two opposite effects on the behavior of the Brownian motions. On the one hand, the ball $B(0,\varepsilon)$ is not completely covered by the Wiener sausages, hence none of the Brownian motions can visit the ball $B(0,1-\varepsilon)$. On the other hand, the ball $B(0,\varepsilon)$ must be intersected by the Wiener sausages from several different directions in order to disconnect the vacant set in several components, hence the Brownian motions either have to visit the ball $B(0,1+\varepsilon)$ on very peculiar locations (e.g.\ on the opposite sides near the boundary) or spend a long time in the ball. The interplay of these repulsive and attractive effects is essential to get good  decay. (In fact, to get any better decay than $\varepsilon^K$.) 
In Lemma~\ref{l:uniqueness-hemiball}, we make a crucial observation that the nonuniqueness event in  \eqref{intro:uniqueness-severalballs} implies the (geometrically much more transparent) event that (a) none of Brownian motions visit $B(0,1-\varepsilon)$ and (b) each $\varepsilon$-hemiball $\{x\in B(0,1+\varepsilon)\,:\,\langle x,e\rangle\geq -\varepsilon\}$ (see Figure~\ref{fig:hemiball}), where $\|e\|=1$, is visited by at least one of the Brownian motions. (This is actually the only step in the proof of \eqref{intro:uniqueness-BI}, where we essentially use that Wiener sausages are unions of Euclidean balls.)

\smallskip

Finally, in Theorem~\ref{thm:hitting-all-hemiballs}, we show that 
for independent Brownian motions $W^{(1)},\ldots, W^{(K)}$ in $\R^d$ started from uniform points in $\partial B(0,1)$ and for any $\varepsilon\in(0,1)$, 
\begin{equation}\label{intro:hitting-all-hemiballs}
\mathsf P\left[
\begin{array}{c}
\text{each $\varepsilon$-hemiball $\{x\in B(0,1):\langle x,e\rangle\geq -\varepsilon\}$, for $\|e\|=1$,}\\ 
\text{is visited by at least one of the Brownian motions and}\\
\text{none of the Brownian motions visits $B(0,1-\varepsilon)$}
\end{array}
\right]\leq C\log^m\big(\tfrac1\varepsilon\big)\varepsilon^{d+1},
\end{equation}
for some dimension dependent constants $C$ and $m$, which---together with the observation of Lemma~\ref{l:uniqueness-hemiball} and scaling invariance of the Brownian motion---implies \eqref{intro:uniqueness-severalballs}. Statement \eqref{intro:hitting-all-hemiballs} is the main technical novelty of this paper. To prove it, we show that the event in \eqref{intro:hitting-all-hemiballs} implies the existence of a certain finite cascade of Brownian excursions from $\partial B(0,1)$ to $\partial B(0,1+r_l)$ for some (random) decreasing sequence of ranges $r_l$, which is defined recursively in terms of the distance to the origin from the affine hull spanned by the starting points of Brownian excursions from the previous iteration, see Sections~\ref{sec:cascade} and \ref{sec:cascade-lexicographic}.

\smallskip

We would like to finish this discussion by mentioning a continuum percolation model introduced in \cite{EMP17,EP16}, which---despite exponential decay of correlations---is very similar on microscales to Brownian interlacements. The model is defined as a Poisson cloud of finite-time Wiener sausages in $\R^d$. 
The two papers only consider the set occupied by the Wiener sausages and do not need to face challenges coming from microscopic scales. 
One may ask if the vacant set in this model contains at most one infinite connected component. Because of the finite range of the sausages, one can justify local modifications in a more direct way (e.g.\ similar to the vacant set of the Boolean model in \cite{MR-Uniqueness}), but one may need to prove a statement like \eqref{intro:expected-number-components} for this model, to rule out the possibility of infinitely many infinite components. 
This may be one incentive to extend our results to ensembles of finite-time Wiener sausages. 

\smallskip

Let us outline how the article is organized. We have divided the article into three parts, which can essentially be read independently of each other. The first part is  Section~\ref{sec:hitting-hemiballs}, where we prove \eqref{intro:hitting-all-hemiballs} (see Theorem~\ref{thm:hitting-all-hemiballs}); Sections~\ref{sec:BM-basics} and \ref{sec:r-excursions} contain some preliminaries on Brownian motion; the key construction of the cascade of Brownian excursions is described in Sections~\ref{sec:cascade} and \ref{sec:cascade-lexicographic}; the main result about the cascade is Lemma~\ref{l:probability:tau}, which is proven in Section~\ref{sec:probability-tau}; some auxiliary results about perturbation of affine hulls and hitting probabilities for Brownian motions are proven, respectively, in Sections~\ref{sec:aux-affine-hulls} and \ref{sec:aux-BM-affine-hulls}.
The second part is Section~\ref{sec:uniqueness-WS}, where we prove a more general version of microscopic unqueness \eqref{intro:uniqueness-severalballs} for several well separated balls (see Theorem~\ref{thm:uniqueness-severalballs}); in Section~\ref{sec:uniqueness-WS-oneball-hemiballs}, we prove the main reduction of microscopic nonuniqueness \eqref{intro:uniqueness-severalballs} to the event about hitting all $\varepsilon$-hemiballs (see Lemma~\ref{l:uniqueness-hemiball}); in Sections~\ref{sec:nonuniqueness-reduction-to-hemiballs} and \ref{sec:uniqueness-excursions}, we prove microscopic nonuniqueness for one $\varepsilon$-ball using the result of Theorem~\ref{thm:hitting-all-hemiballs}; finally, in Section~\ref{sec:uniqueness-severalballs-proof}, we prove Theorem~\ref{thm:uniqueness-severalballs} about the microscopic nonuniqueness in several $\varepsilon$-balls. The third part is Sections~\ref{sec:BI}, \ref{sec:connectivity-BI} and \ref{sec:uniqueness-proof}; Section~\ref{sec:BI} contains definition of Brownian interlacements point process as a Poisson process on the space of doubly infinite paths as well as some useful sampling procedure (see \eqref{eq:sample-1} and \eqref{eq:sample-2}); Section~\ref{sec:connectivity-BI} contains a refinement of the result of Li \cite{Li-BI} about connectedness of Brownian interlacements (see Proposition~\ref{prop:BI-connectivity-int} and Corollary~\ref{cor:condBI-connectivity-int}); in Section~\ref{sec:uniqueness-proof} we prove  \eqref{intro:uniqueness-BI} (see Theorem~\ref{thm:uniqueness-BI}); we prove \eqref{intro:expected-number-components} using Theorem~\ref{thm:uniqueness-severalballs} in Section~\ref{sec:expected-number} (see Theorem~\ref{thm:expected-number-components}), which can be read independently of Sections~\ref{sec:BI}, \ref{sec:connectivity-BI} and the rest of Section~\ref{sec:uniqueness-proof}.

\smallskip

Throughout the paper, we write $B(x,r)$ for the closed Euclidean ball in $\R^d$ of radius $r$ centered in $x$ and $B(r)$ for $B(0,r)$. For $K\subset\R^d$, we write $B(K,r)$ for $\bigcup\limits_{x\in K}B(x,r)$.

\section{Hitting all hemiballs of a unit ball with several Brownian motions}\label{sec:hitting-hemiballs}

In this section we prove that several independent Brownian motions started on the boundary of $B(1)$ are unlikely to jointly visit all hemiballs of $B(1)$ while none of them visits a slightly smaller ball $B(1-\varepsilon)$. This is the main technical novelty of our paper.
In Section~\ref{sec:uniqueness-WS}, we relate this event to the nonuniqueness of connected components of the vacant set of several Wiener sausages in a tiny ball $B(\varepsilon)$ (Lemma~\ref{l:uniqueness-hemiball}) and prove that nonuniqueness is unlikely to occur in several separated small balls (Theorem~\ref{thm:uniqueness-severalballs}). This allows us in the end to prove that the expected number of large connected components of the vacant set of Brownian interlacements that intersect a ball is finite (Theorem~\ref{thm:expected-number-components}), which is an important ingredient in the proof that the number of infinite connected components in the vacant set of Brownian interlacements is not infinite (Proposition~\ref{prop:N=01}).

\begin{definition}\label{def:hemiball}
For $r>0$, $\delta>0$ and $e\in\R^d$ with $\|e\|=1$, let
\[
A_{e,\delta}(r) = \{x\in B(r)\,:\,\langle x,e\rangle\geq -\delta\}
\]
(see Figure~\ref{fig:hemiball}) 
and set $\mathcal A_{\delta}(r) = \big\{A_{e,\delta}(r)\,:\,e\in\R^d,\,\|e\|=1\big\}$. 
We call every $A_{e,\delta}(r)$ a \emph{$\delta$-hemiball of $B(r)$}. 
\end{definition}

\begin{figure}[!tp]
\centering
\resizebox{5cm}{!}{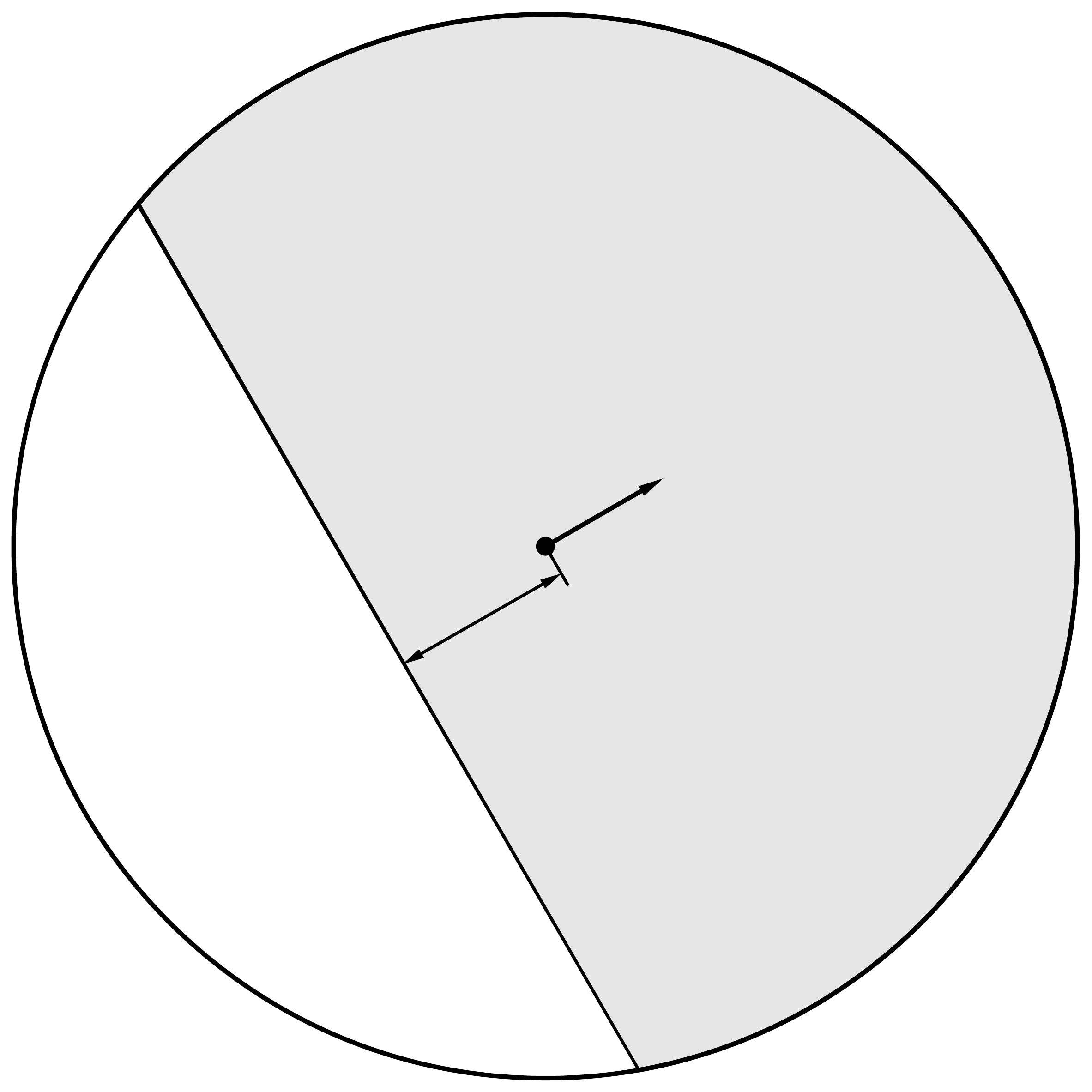}
\caption{$A_{e,\delta}(r)= \{x\in B(r)\,:\,\langle x,e\rangle\geq -\delta\}$}
\label{fig:hemiball}
\end{figure}

\begin{theorem}\label{thm:hitting-all-hemiballs} Let $d\geq 3$ and $\varepsilon\in(0,1)$. For $1\leq K\leq d$, let $W^{(1)},\ldots, W^{(K)}$ be independent Brownian motions in $\R^d$, started from uniform points on $\partial B(1)$. Let $E$ be the event that every $\varepsilon$-hemiball in $\mathcal A_\varepsilon(1)$ is visited by at least one of the Brownian motions. Let $F$ be the event that none of the Brownian motions visits $B(1-\varepsilon)$. 
Then there exist $C=C(d)$ and $m=m(d)$, such that for all $\varepsilon$,
\[
\mathsf P[E\cap F]\leq C\log^m(\tfrac1\varepsilon)\, \varepsilon^{d+1}.
\]
\end{theorem}

The proof of Theorem~\ref{thm:hitting-all-hemiballs} proceeds by a delicate analysis of the interplay between the \emph{number} of Brownian excursions from $\partial B(1)$ to $\partial B(1+r)$ (for a (random) sequence of ranges $r$), their \emph{starting locations} and their \emph{radii}: On the one hand, each excursion comes with a cost of $\frac{\varepsilon}{r}$, since it has to avoid $B(1-\varepsilon$), so the number of excursions cannot be too big; on the other hand, if the number of excursions is not too big, for event $E$ to occur, either they have to start from a rather degenerate location (e.g.\ in case of two excursions, direct opposite to each other on $\partial B(1)$) or some of them have to have large radius, both of which are unlikely. In Section~\ref{sec:r-excursions}, we identify a likely event $G$, which allows us to assume that all the excursions have suitably small radii (see Lemma~\ref{l:eventG}); this argument is standard. The other two features (number and location) are controlled in Section~\ref{sec:cascade} by constructing a cascade of excursions with random ranges $r_l$; this approach is new and allows to implicitly assess the cost of a large number of excursions resp.\ peculiar location of their starting points; the main statement there is Lemma~\ref{l:probability:tau}.

\smallskip

In fact, if the starting points of the Brownian motions are well-separated and their affine hull is within distance $\varepsilon$ from $0$, which occurs with probability $\geq c\varepsilon^{(d+1-K)_+}$, then, conditioned on $F$, $E$ occurs with uniformly positive probability, hence $\mathsf P[E\cap F]\geq (c \varepsilon)^{(d+1)\vee K}$.

\subsection{Some notation and basic facts about Brownian motion}\label{sec:BM-basics}

We denote by $W$ a Brownian motion in $\R^d$ ($d\geq 3$) and by $\mathsf P_x$ the law of $W$ with $\mathsf P_x[W_0=x]=1$. 

For $r>0$, let $T_r = \inf\{t\geq 0\,:\, W_t\in \partial B(r)\}$ be the first entrance time of $W$ in $\partial B(r)$. We will use the following classical results:
\begin{lemma}\label{l:BM-hitting}
For any $r<R$ and $x\in\R^d$ with $r<\|x\|<R$, 
\[
\mathsf P_x[T_R<T_r] = \frac{r^{2-d} - \|x\|^{2-d}}{r^{2-d} - R^{2-d}}.
\]
In particular, $\mathsf P_x[T_r = +\infty] = 1- \frac{\|x\|^{2-d}}{r^{2-d}}$.
\end{lemma}
\begin{lemma}\label{l:poisson-formula}[Poisson's formula]
For any $r>0$, $B\in\mathscr B(\partial B(r))$ and $x\notin \partial B(r)$, 
\[
\mathsf P_x[W_{T_r}\in B] = \frac{r^{d-2}}{|\partial B(r)|}\int\limits_B\frac{|r^2-\|x\|^2|}{\|x-y\|^d}dy.
\]
\end{lemma}
Lemma~\ref{l:BM-hitting} can be found e.g.\ in \cite[Theorem~3.18 and Corollary~3.19]{MP-BM-book}, Lemma~\ref{l:poisson-formula} follows e.g.\ from \cite[Theorem~3.44]{MP-BM-book}.

\subsection{\texorpdfstring{$r$}{r}-excursions and their radius}\label{sec:r-excursions}
For a Brownian motion $W$ in $\R^d$ ($d\geq 3$) started on $\partial B(1)$ and $r>0$, 
consider the times of successive revisits of $B(1)$ after leaving $B(1+r)$:
\[
\tau_1 = 0,\,\, \eta_1 = \inf\{t>0\,:\,W_t\in\partial B(1+r)\},
\]
and for $k\geq 1$, 
\[
\tau_k = \inf\{t>\eta_{k-1}\,:\,W_t\in\partial B(1)\},\,\,
\eta_k = \inf\{t>\tau_k\,:\,W_t\in\partial B(1+r)\},
\]
where $\inf\emptyset = +\infty$. 

The fragments $\{W_t, t\in[\tau_k,\eta_k]\}_{k\geq 1}$ of the Brownian motion $W$ are called the \emph{$r$-excursions} of $W$. The \emph{radius} of $r$-excursion $W_{[\tau_k,\eta_k]}$ is defined as $\max\limits_{t\in[\tau_k,\eta_k]} \|W_t - W_{\tau_k}\|$.

\begin{lemma}\label{l:diameter-excursion}
Let $\varepsilon\in(0,1)$. Let $A$ be the event that, for some $r\in[\varepsilon,1]$, there is a $r$-excursion of $W$ with radius $\geq (\log\varepsilon)^2\, r$. 
Then there exists $C=C(d)$ such that for all $\varepsilon$ and $x\in\partial B(1)$, 
\[
\mathsf P_x\big[A,\,T_{1-\varepsilon}=+\infty\big]< C\,\varepsilon^{d+1}.
\]
\end{lemma}
\begin{proof}
It suffices to prove the lemma for $\varepsilon\leq \varepsilon_0(d)$, for some $\varepsilon_0(d)>0$ small enough. 

Let $A_r$ be the event that there is a $r$-excursion of $W$ with radius $\geq\frac12(\log\varepsilon)^2r$. Let $\mathcal R = \big\{\frac1{2^k}\,:\,0\leq k\leq \log_2(\frac1\varepsilon)\big\}$. Note that for every $r\in[\varepsilon,1]$ there exists $\rho\in\mathcal R$ such that $\frac12\rho<r\leq \rho$; in particular, if there is a $r$-excursion with radius $\geq(\log\varepsilon)^2r$, then $A_\rho$ occurs. Thus, 
\[
\mathsf P_x\big[A,\,T_{1-\varepsilon}=+\infty\big] \leq 
\sum\limits_{\rho\in\mathcal R}\mathsf P_x\big[A_\rho, T_{1-\varepsilon}=+\infty\big]\leq \big(1+\log_2\big(\tfrac1\varepsilon\big)\big)\,\sup\limits_{\rho\in\mathcal R}\mathsf P_x\big[A_\rho, T_{1-\varepsilon}=+\infty\big].
\]
Let $\rho\in\mathcal R$. Note that for any $z\in B(1+\rho)\setminus B(1-\varepsilon)$, the probability that a Brownian motion started at $z$ will exit from $B(1+\rho)$ before leaving $B(z,2\rho\sqrt{d})$ is at least $\frac{1}{2d}$.
Let $\zeta = \inf\{t\geq 0\,:\,\|W_t - W_0\|\geq \frac12(\log\varepsilon)^2\rho\}$. By considering successive exits of the Brownian motion from balls $B(W_0,2\sqrt{d}\rho k)$, for $1\leq k\leq \tfrac{1}{4\sqrt{d}}(\log\varepsilon)^2$ and using the strong Markov property, we obtain that
\[
\mathsf P_y[\zeta\leq \min(T_{1+\rho},T_{1-\varepsilon})]\leq \big(1-\tfrac{1}{2d}\big)^{\lfloor\tfrac{1}{4\sqrt{d}}(\log\varepsilon)^2\rfloor},\quad y\in\partial B(1).
\]
Furthermore, by Lemma~\ref{l:BM-hitting}, for any $\rho\geq \varepsilon$, 
\[
\mathsf P_y[T_{1+\rho}<T_{1-\varepsilon}] 
\leq \mathsf P_y[T_{1+\varepsilon}<T_{1-\varepsilon}]
\leq \tfrac23,\quad y\in\partial B(1),
\]
for all $\varepsilon\leq \varepsilon_1$, where $\varepsilon_1=\varepsilon_1(d)>0$ is small enough. 

Thus, by decomposing event $A_\rho$ according to the first $\rho$-excursion with big radius and using the strong Markov property, we obtain that 
\[
\mathsf P_x\big[A_\rho, T_{1-\varepsilon}=+\infty\big]\leq 
\sum\limits_{n=1}^\infty \big(\tfrac23\big)^{n-1}\big(1-\tfrac{1}{2d}\big)^{\lfloor\tfrac{1}{4\sqrt{d}}(\log\varepsilon)^2\rfloor}
= 3\big(1-\tfrac{1}{2d}\big)^{\lfloor\tfrac{1}{4\sqrt{d}}(\log\varepsilon)^2\rfloor}
<\frac{\varepsilon^{d+1}}{1+\log_2\big(\tfrac1\varepsilon\big)},
\]
for all $\varepsilon\leq \varepsilon_2$, where $\varepsilon_2=\varepsilon_2(d)>0$ is small enough. 
The proof is completed. 
\end{proof}

\medskip

In the proof of Theorem~\ref{thm:hitting-all-hemiballs} we will use the following result, which is immediate from Lemma~\ref{l:diameter-excursion}. 
\begin{lemma}\label{l:eventG}
For $1\leq K\leq d$, let $W^{(1)},\ldots, W^{(K)}$ be independent Brownian motions in $\R^d$ ($d\geq 3$) started on $\partial B(1)$. For $\varepsilon\in(0,1)$, let $G$ be the event that for every $r\in[\varepsilon,1]$, every $r$-excursion of every Brownian motion $W^{(k)}$ has radius $<(\log\varepsilon)^2r$ and let $F$ be the event that none of the Brownian motions $W^{(1)},\ldots, W^{(K)}$ visits $B(1-\varepsilon)$. Then there exists $C=C(d)$ such that for all $\varepsilon$, 
\[
\mathsf P[G^c\cap F]\leq C\varepsilon^{d+1}.
\]
\end{lemma}
By Lemma~\ref{l:eventG}, in the proof of Theorem~\ref{thm:hitting-all-hemiballs} we may assume that event $G$ occurs. 

\subsection{Cascade of excursions}\label{sec:cascade}

Let $W^{(1)}, \ldots, W^{(K)}$ be Brownian motions started on $\partial B(1)$. In what follows, we define recursively a family of $r$-excursions for these Brownian motions for a decreasing sequence of (random) ranges $r=r_l$. This construction is key to the estimation of the probability of the event $E\cap F\cap G$, see Lemma~\ref{l:EG-Lleqd}.

\medskip

Let $\varepsilon\in(0,1)$ be such that $\log(\tfrac1\varepsilon)\geq 1$ and define 
\begin{equation}\label{def:gamma}
\gamma=\gamma(\varepsilon) = (\log\varepsilon)^2 + 1
\end{equation}
and 
\[
\mathcal R = \big\{\tfrac1{2^k}\,:\,0\leq k\leq \log_2(\tfrac1\varepsilon)\big\}.
\]

\bigskip

\emph{Step $0$.} For $1\leq k\leq K$, let 
\[
\overline P^{(k)}_0 = W^{(k)}_0 \in\partial B(1)
\]
be the starting point of the Brownian motion $W^{(k)}$. We denote by $H_0$ the affine hull of these $K$ points,
\[
H_0 = \mathrm{aff}\big\{\overline P^{(k)}_0\,:\, 1\leq k\leq K\big\},
\]
and by $d_0$ the Euclidean distance from $H_0$ to the origin. (Note that $d_0$ is random.) For convenience, we also write $\overline N_0$ instead of $K$, when referring to the number of points $\overline P^{(k)}_0$.

\smallskip

If $\overline N_0\geq d+1$ or $d_0<2\gamma\varepsilon$, we stop the procedure; otherwise, proceed to Step $1$. 

\bigskip

\emph{Step $1$.} Let 
\[
r_1 = \max\big\{r\in\mathcal R\,:\, r\leq \gamma^{-1} d_0\big\}
\]
and consider the $r_1$-excursions of the Brownian motions.\footnote{Since $d_0\geq 2\gamma\varepsilon$, there is $r\in\mathcal R$ such that $r\leq \gamma^{-1}d_0$, so $r_1$ is well defined.} For $1\leq k\leq K$, let 
\[
\overline \tau^{(k)}_{1,1} = 0,\,\, \overline \eta^{(k)}_{1,1} = \inf\{t>0\,:\,W^{(k)}_t\in\partial B(1+r_1)\},
\]
and for $i\geq 2$, 
\[
\overline \tau^{(k)}_{1,i} = \inf\{t>\overline\eta^{(k)}_{1,i-1}\,:\,W^{(k)}_t\in\partial B(1)\},\,\,
\overline\eta^{(k)}_{1,i} = \inf\{t>\overline\tau^{(k)}_{1,i}\,:\,W^{(k)}_t\in\partial B(1+r_1)\},
\]
where $\inf\emptyset = +\infty$. We denote by $\overline N^{(k)}_1 = \sup\{i\geq 0\,:\,\overline \tau^{(k)}_{1,i}<\infty\}$ the number of $r_1$-excursions for $W^{(k)}$ and by 
\[
\overline P^{(k)}_{1,i} = W^{(k)}_{\overline \tau^{(k)}_{1,i}},\quad 1\leq i\leq \overline N^{(k)}_1,
\]
the starting points of the $r_1$-excursions for $W^{(k)}$. 
Let 
\[
H_1 = \mathrm{aff}\big\{\overline P^{(k)}_{1,i_k}\,:\,1\leq k\leq K,\,1\leq i_k\leq \overline N^{(k)}_1\big\}
\]
be the affine hull of all the starting points of the $r_1$-excursions and denote by $d_1$ the Euclidean distance from $H_1$ to the origin. 
Finally, let $\overline N_1 = \sum\limits_{k=1}^N \overline  N^{(k)}_1$ be the total number of $r_1$-excursions. 

\smallskip

If $\overline N_1\geq d+1$ or $d_1<2\gamma\varepsilon$, we stop the procedure; otherwise, 
proceed to Step $2$.

\bigskip

\emph{Step $l$} ($l\geq 2$). If the construction proceeds to Step $l$, then we have defined in the preceding step the radius $r_{l-1}$, the affine hull $H_{l-1}$ spanned by all the starting points of the $r_{l-1}$-excursions, its distance $d_{l-1}$ to the origin, and the total number $\overline N_{l-1}$ of $r_{l-1}$-excursions; and, furthermore, 
$\overline N_{l-1}\leq d$ and $d_{l-1}\geq 2\gamma\varepsilon$.

We define 
\[
r_l = \max\big\{r\in\mathcal R\,:\, r\leq \gamma^{-1} d_{l-1}\big\}
\]
and consider the $r_l$-excursions of the Brownian motions.\footnote{Note that $r_l$ is well defined when $d_{l-1}\geq 2\gamma\varepsilon$ and $r_l\leq r_{l-1}$.} Similarly to Step $1$, we introduce the following random variables:
\begin{itemize}
\item
the successive revisits of $B(1)$ after leaving $B(1+r_l)$, 
\begin{equation}\label{def:overline-tau-eta}
\begin{aligned}
\overline \tau^{(k)}_{l,1} &= 0,  &\overline \eta^{(k)}_{l,1} &= \inf\{t>0\,:\,W^{(k)}_t\in\partial B(1+r_l)\},\\
\overline \tau^{(k)}_{l,i} &= \inf\{t>\overline\eta^{(k)}_{l,i-1}\,:\,W^{(k)}_t\in\partial B(1)\},
&\overline\eta^{(k)}_{l,i} &= \inf\{t>\overline\tau^{(k)}_{l,i}\,:\,W^{(k)}_t\in\partial B(1+r_l)\},
\end{aligned}
\end{equation}
for $i\geq 2$, where $\inf\emptyset = +\infty$; 

\item
the number $\overline N^{(k)}_l$ of $r_l$-excursions for $W^{(k)}$ and the total number $\overline N_l$ of $r_l$-excursions; 

\item
the starting points of the $r_l$-excursions,
\[
\overline P^{(k)}_{l,i} = W^{(k)}_{\overline \tau^{(k)}_{l,i}},\quad 1\leq i\leq \overline N^{(k)}_l;
\]

\item
the affine hull of the starting points of $r_l$-excursions,
\[
H_l = \mathrm{aff}\big\{\overline P^{(k)}_{l,i_k}\,:\,1\leq k\leq K,\,1\leq i_k\leq \overline N^{(k)}_l\big\};
\]

\item
the Euclidean distance $d_l$ from $H_l$ to the origin.\footnote{Since each of the starting points of $r_{l-1}$-excursions is also a starting point of a $r_l$-excursion, we have that $H_0\subseteq H_1\subseteq H_2\subseteq \ldots$ and, thus, $d_0\geq d_1\geq d_2\geq\ldots$.}
\end{itemize}

\smallskip

If the total number of excursions $\overline N_l\geq d+1$ or the distance $d_l<2\gamma\varepsilon$, we stop the procedure; otherwise, proceed to the next step. 

\bigskip

We denote by $L$ the \emph{last step} of the above construction, 
\begin{equation}\label{def:L}
L = \inf\big\{l\geq 0\,:\, \overline N_l\geq d+1\text{ or }d_l<2\gamma\varepsilon\big\}\in [0,+\infty].
\end{equation}
Note that the construction can get stuck in an infinite loop, resulting in $L=+\infty$, if for some $l$, $r_l=r_{l-1}$. In fact, the following holds. 
\begin{lemma}\label{l:infinite-loop-L}
Either $L=+\infty$ or $\overline N_l\geq \overline N_{l-1}+1$ for all $1\leq l\leq L$. In particular, if $L<+\infty$ then $L\leq d$. 
\end{lemma}
\begin{proof}
If there exists $l\leq L$ such that $\overline N_l=\overline N_{l-1}$, then $H_l=H_{l-1}$, hence $r_{l+1}=r_l$ and we have $L=+\infty$. The second statement is obvious from the definition of $L$.
\end{proof}

\smallskip

A crucial observation for the proof of Theorem~\ref{thm:hitting-all-hemiballs} is that event $E\cap G$ implies $L\leq d$:

\begin{lemma}\label{l:EG-Lleqd}
If events $E$ from Theorem~\ref{thm:hitting-all-hemiballs} and $G$ from Lemma~\ref{l:eventG} occur, then 
$L\leq d$.
\end{lemma}
\begin{proof}
Assume that $L=+\infty$. By Lemma~\ref{l:infinite-loop-L} there exists $l$ such that $\overline N_l=\overline N_{l-1}$. Then $H_l=H_{l-1}$; in particular, every $r_l$-excursion begins in $H_{l-1}$. 

Let $x\in H_{l-1}$ be the closest point of $H_{l-1}$ to the origin and let $e$ be the unit vector in $\R^d$ collinear with $0$ and $x$ and oriented opposite to $x$. Let $A_{e,\varepsilon}(1)$ be the corresponding $\varepsilon$-hemiball (see Definition~\ref{def:hemiball} and Figure~\ref{fig:hemiball}). Note that the distance between $A_{e,\varepsilon}(1)$ and $H_{l-1}$ is 
\[
d_{l-1}-\varepsilon \geq \gamma r_l - \varepsilon \geq (\gamma - 1)\,r_l\stackrel{\eqref{def:gamma}}=(\log\varepsilon)^2r_l,
\]
where in the second inequality we used that $r_l\in\mathcal R$ and thus $r_l\geq \varepsilon$. 

\smallskip

If event $E$ occurs, then $A_{e,\varepsilon}(1)$ has to be visited by one of the Brownian motions, hence by one of the $r_l$-excursions. However, all the $r_l$-excursions start in $H_{l-1}$, that is at distance $\geq(\log\varepsilon)^2r_l$ from $A_{e,\varepsilon}(1)$. Thus, at least one of the $r_l$-excursions must have radius $\geq(\log\varepsilon)^2r_l$, hence event $G$ does not occur. 

We have shown that $L=+\infty$ implies that either $E$ or $G$ does not occur. Thus, by Lemma~\ref{l:infinite-loop-L}, $E\cap G$ implies $L\leq d$. The proof is completed. 
\end{proof}

\subsection{Cascade of excursions II, lexicographic order}\label{sec:cascade-lexicographic}

By Lemmas~\ref{l:eventG} and \ref{l:EG-Lleqd}, to prove Theorem~\ref{thm:hitting-all-hemiballs} it suffices to estimate the probability $\mathsf P[F,G,L\leq d]$. For that, it will be important to refine the construction of Section~\ref{sec:cascade} and distinguish the starting points of $r_l$-excursions which are \emph{new} with respect to the starting points of the $r_{l-1}$-excursions.  

We define 
\[
\tau^{(k)}_{1,i} = \overline \tau^{(k)}_{1,i+1},\quad i\geq 1,
\]
and for $l\geq 2$, 
\begin{align*}
\tau^{(k)}_{l,1} &= \inf\big\{\overline\tau^{(k)}_{l,j}>0\,:\,\overline\tau^{(k)}_{l,j}\neq \overline \tau^{(k)}_{l-1,j'}\text{ for all }1\leq j'\leq \overline N^{(k)}_{l-1}\big\}\\
\tau^{(k)}_{l,i} &= \inf\big\{\overline\tau^{(k)}_{l,j}>\tau^{(k)}_{l,i-1}\,:\,\overline\tau^{(k)}_{l,j}\neq \overline \tau^{(k)}_{l-1,j'}\text{ for all }1\leq j'\leq \overline N^{(k)}_{l-1}\big\}\quad (i\geq 2),
\end{align*}
where $\inf\emptyset=+\infty$. 

We denote by $N^{(k)}_l = \sup\{i\geq 1\,:\,\tau^{(k)}_{l,i}<+\infty\}$ ($=0$ when $\tau^{(k)}_{l,1}=+\infty$) the difference between the number of $r_l$- and $r_{l-1}$-excursions for $W^{(k)}$ ($N^{(k)}_1$ is one less the number of $r_1$-excursions) and let $N_l = \sum\limits_{k=1}^K N^{(k)}_l$.

The corresponding new starting points of the $r_l$-excursions are denoted by  
\[
P^{(k)}_{l,i} = W^{(k)}_{\tau^{(k)}_{l,i}},\quad 1\leq i\leq N^{(k)}_l.
\]
For convenience, we also define $P^{(k)}_{0,1} = W^{(k)}_0$, $N^{(k)}_0 = 1$ and $N_0 = K$. 

\medskip

We would like to view the points $P^{(k)}_{l,i}$ in  a specific order. For this, we consider the index set  
\[
\mathcal I = \{(l,k,i)\,:\,l\geq 0, 1\leq k\leq K,\,1\leq i\leq N^{(k)}_l\}
\]
and denote by $\prec$ the lexicographic order on $\mathcal I$.\footnote{$(l,k,i)\prec(l',k',i')$ if (a) $l<l'$ or (b) $l=l'$, $k<k'$ or (c) $l=l'$, $k=k'$, $i<i'$.}

Let 
\[
H^{(k)}_{l,i} = \mathrm{aff}\big\{P^{(k')}_{l',i'}\,:\,(l',k',i')\preceq(l,k,i)\big\}\quad\text{and}\quad \check H^{(k)}_{l,i} =  \mathrm{aff}\big\{P^{(k')}_{l',i'}\,:\,(l',k',i')\prec(l,k,i)\big\}
\]
and denote by $d^{(k)}_{l,i}$ the Euclidean distance from $H^{(k)}_{l,i}$ to the origin.

Notice that $H^{(K)}_{l,N^{(K)}_l} = H_l$ and $d^{(K)}_{l,N^{(K)}_l} =d_l$.

There is a natural one-to-one correspondence between elements of $\mathcal I$ and $\N$ given by the map
\begin{equation}\label{def:Nkli}
N^{(k)}_{l,i} = \big|\{(l',k',i')\,:\,(l',k',i')\preceq(l,k,i)\}\big|.
\end{equation}
It will be convenient to define the following stopping rule $\tau$, which is finer than the stopping rule defined by the variable $L$ in \eqref{def:L}.
\begin{definition}\label{def:tau}
Let $\tau$ be the smallest $N^{(k)}_{l,i}(\geq 2)$, such that one of the following events occurs:
\begin{enumerate}
\item
$d\big(P^{(k)}_{l,i},\check H^{(k)}_{l,i}\big)<\varepsilon$;
\item
$d^{(k)}_{l,i}<2\gamma\varepsilon$ (recall \eqref{def:gamma} for definition of $\gamma$);
\item
$N^{(k)}_{l,i} = d+1$;
\end{enumerate}
and $\tau=+\infty$ otherwise.\footnote{$\tau=+\infty$ if the construction gets stuck in an infinite loop before any of the events (1)-(3) occurs.}
\end{definition}

By the definition \eqref{def:L} of $L$, $L\leq d$ implies that $\tau\leq d+1$. Thus, by Lemmas~\ref{l:eventG} and \ref{l:EG-Lleqd}, Theorem~\ref{thm:hitting-all-hemiballs} follows from the following lemma.
\begin{lemma}\label{l:probability:tau}
For $1\leq K\leq d$, let $W^{(1)}, \ldots, W^{(K)}$ be independent Brownian motions started from uniform points on $\partial B(1)$. Let $F$ be the event from Theorem~\ref{thm:hitting-all-hemiballs} and $G$ the event from Lemma~\ref{l:eventG}. There exist $C=C(d)$ and $m=m(d)$, such that for all $\varepsilon$,
\[
\mathsf P[F,G,\tau=t]\leq C\log^m(\tfrac1\varepsilon)\, \varepsilon^{d+1},\quad 2\leq t\leq d+1.
\]
\end{lemma}

\subsection{Proof of Lemma~\ref{l:probability:tau}}\label{sec:probability-tau}
We first specify the triple $(\ell,\kappa,\iota)$ such that $\tau=N^{(\kappa)}_{\ell,\iota}$. By the union bound, 
\[
\mathsf P[F,G,\tau=t]\leq \sum\limits_{\ell=0}^d\sum\limits_{\kappa=1}^K\sum\limits_{\iota\geq 1}\mathsf P\big[F,G,\tau=t=N^{(\kappa)}_{\ell,\iota}\big]\leq d(d+1)^2\,\sup\limits_{\ell,\kappa,\iota}\mathsf P\big[F,G,\tau=t=N^{(\kappa)}_{\ell,\iota}\big].
\]
Thus, it suffices to bound $\mathsf P\big[F,G,\tau=t=N^{(\kappa)}_{\ell,\iota}\big]$ for each fixed triple $(\ell,\kappa,\iota)$.

\smallskip

Since the excursions in the cascade are not defined chronologically, in order to be able to apply the strong Markov property, we will decompose the event according to the number of different excursions, their ranges $r_l$, the locations of the points $P^{(k)}_{l,i}$ and the distances of respective affine hulls to the origin. 

\smallskip

Recall that $N^{(k)}_l$ denotes the difference between the number of $r_l$- and $r_{l-1}$-excursions for Brownian motion $W^{(k)}$. Given $\tau = N^{(\kappa)}_{\ell,\iota}$, we also define for each $1\leq k\leq K$, 
\[
\widetilde N^{(k)}_\ell = \big|\{(\ell,k,i)\,:\,1\leq i\leq N^{(k)}_\ell\text{ such that }(\ell,k,i)\preceq(\ell,\kappa,\iota)\}\big|.
\]
When $\ell\geq 1$, $\widetilde N^{(k)}_\ell$ is the number of new starting-points of $r_\ell$-excursions of the $k$-th Brownian motion with indices $(\ell,k,i)$ not exceeding $(\ell,\kappa,\iota)$ lexicographically. When $\ell=0$, $\widetilde N^{(k)}_\ell=1$ for the first $\tau$ Brownian motions and $=0$ for the remaining $K-\tau$ Brownian motions. 
We have 
\begin{multline*}
\mathsf P\big[F,G,\tau=t=N^{(\kappa)}_{\ell,\iota}\big]\\
\leq
\sum\limits_{\{n^{(k)}_{l,i}\}}\sum\limits_{\{\widetilde n^{(k)}_{\ell,i}\}}\mathsf P\big[F,G,\tau=t=N^{(\kappa)}_{\ell,\iota},N^{(k)}_l = n^{(k)}_l,\widetilde N^{(k)}_\ell = \widetilde n^{(k)}_\ell,\text{ for all }l,k\big]
\end{multline*}
where the sum is over all $n^{(k)}_{l,i}\in \N_0$, $\widetilde n^{(k)}_{\ell,i}\in\N_0$ satisfying $\sum\limits_{k=1}^K\big(\sum\limits_{l=0}^{\ell-1} n^{(k)}_l + \widetilde n^{(k)}_\ell\big)=t$. 

Since $0\leq \ell\leq d$, $1\leq K\leq d$ and $t\leq d+1$, the above sum is bounded by 
\[
(d+2)^{d(d+1)}\,\sup\limits_{\{n^{(k)}_{l,i}\},\{\widetilde n^{(k)}_{\ell,i}\}}\mathsf P\big[F,G,\tau=t=N^{(\kappa)}_{\ell,\iota},N^{(k)}_l = n^{(k)}_l,\widetilde N^{(k)}_\ell = \widetilde n^{(k)}_\ell,\text{ for all }l,k\big].
\]
From now on we fix $\{n^{(k)}_l\}$ and $\{\widetilde n^{(k)}_\ell\}$ and use the convention (without further mentioning) that we only consider the triples $(l,k,i)$, whose parameters satisfy $0\leq l\leq \ell$, $1\leq k\leq K$ and $1\leq i\leq n^{(k)}_l$ resp.\ $1\leq i\leq \widetilde n^{(k)}_\ell$. 

\smallskip

Next, we pin down the locations of the radii $r_l$, the points $P^{(k)}_{l,i}$ and the distances $d^{(k)}_{l,i}$ of the affine hulls to $0$. Let
\[
\overline \varepsilon = \varepsilon^{2d+2}
\]
and fix a set $\mathcal Z\subseteq \partial B(1)$ such that 
\begin{equation}\label{def:mathcal:Z}
\text{for any $x\in\partial B(1)$, $1\leq \big|\{z\in\mathcal Z\,:\, \|z-x\|<\overline\varepsilon\}\big|\leq 3^d$.}
\end{equation}
Let 
\[
\mathcal D = \big\{\tfrac1{2^k}\,:\,0\leq k\leq \log_2(\tfrac1\varepsilon)\big\}. 
\]
By the definition of $\tau$, $d^{(k)}_{l,i}\geq 2\gamma\varepsilon> \varepsilon$ for all $(l,k,i)\prec(\ell,\kappa,\iota)$, thus there exist $\delta^{(k)}_{l,i}\in\mathcal D$, such that $\tfrac12\delta^{(k)}_{l,i}< d^{(k)}_{l,i}\leq \delta^{(k)}_{l,i}$. We obtain
\begin{multline*}
\mathsf P\big[F,G,\tau=t=N^{(\kappa)}_{\ell,\iota},N^{(k)}_l = n^{(k)}_l,\widetilde N^{(k)}_\ell = \widetilde n^{(k)}_\ell,\text{ for all }l,k\big]\\
\leq 
\sum\limits_{\{\rho_l\}}\sum\limits_{\{z^{(k)}_{l,i}\}}\sum\limits_{\{\delta^{(k)}_{l,i}\}}\mathsf P\Big[\begin{array}{l}F,G,\tau=t=N^{(\kappa)}_{\ell,\iota},N^{(k)}_l = n^{(k)}_l,\widetilde N^{(k)}_\ell = \widetilde n^{(k)}_\ell, r_l=\rho_l\\ P^{(k)}_{l,i}\in B^{(k)}_{l,i}, \tfrac12\delta^{(k)}_{l,i}<d^{(k)}_{l,i}\leq \delta^{(k)}_{l,i}\text{ for all }l,k,i\end{array}
\Big],
\end{multline*}
where the sum is over all $\rho_l\in\mathcal R$ with $1\leq l\leq \ell$, $z^{(k)}_{l,i}\in\mathcal Z$ with $(l,k,i)\preceq(\ell,\kappa,\iota)$ and  
$\delta^{(k)}_{l,i}\in\mathcal D$ with $(l,k,i)\prec(\ell,\kappa,\iota)$, such that 
$\delta^{(1)}_{0,1}=1$ and $\delta^{(k)}_{l,i}\geq \delta^{(k')}_{l',i'}$ when $(l,k,i)\prec(l',k',i')$, and where we use the notation $B^{(k)}_{l,i} = B(z^{(k)}_{l,i},\overline\varepsilon)$. 
Furthermore, since $r_l\geq \tfrac12\gamma^{-1}d_{l-1}$ and $d_{l-1} = d^{(K)}_{l-1,n^{(K)}_{l-1}}$, we may assume that 
\begin{equation}\label{eq:relation:rho-delta}
\rho_l\geq \tfrac14\gamma^{-1}\delta^{(k')}_{l',i'}\quad\text{for all }1\leq l\leq \ell\text{ and }(l',k',i')\succeq(l-1,K,n^{(K)}_{l-1}).
\end{equation}

\smallskip

Since $|\mathcal R| = |\mathcal D|=\log_2(\tfrac1\varepsilon)+1$, we obtain that 
\begin{multline}\label{eq:probability-tau:fixed-rho-delta}
\mathsf P\big[F,G,\tau=t=N^{(\kappa)}_{\ell,\iota},N^{(k)}_l = n^{(k)}_l,\widetilde N^{(k)}_\ell = \widetilde n^{(k)}_\ell,\text{ for all }l,k\big]\\
\begin{aligned}
\leq\,\, 
&\big(\log_2(\tfrac1\varepsilon)+1\big)^{2(d+2)^3}\\
&\qquad\sup\limits_{\{\rho_l\}}\sup\limits_{\{\delta^{(k)}_{l,i}\}}\sum\limits_{\{z^{(k)}_{l,i}\}}\mathsf P\Big[\begin{array}{l}F,G,\tau=t=N^{(\kappa)}_{\ell,\iota},N^{(k)}_l = n^{(k)}_l,\widetilde N^{(k)}_\ell = \widetilde n^{(k)}_\ell, r_l=\rho_l\\ P^{(k)}_{l,i}\in B^{(k)}_{l,i}, \tfrac12\delta^{(k)}_{l,i}< d^{(k)}_{l,i}\leq \delta^{(k)}_{l,i}\text{ for all }l,k,i\end{array}
\Big],
\end{aligned}
\end{multline}
From now we also fix $\{\rho_l\}\in\mathcal R^\ell$ and $\{\delta^{(k)}_{l,i}\}\in\mathcal D^{t-1}$ such that $\delta^{(1)}_{0,1}=1$, $\delta^{(k)}_{l,i}\geq \delta^{(k')}_{l',i'}$ when $(l,k,i)\prec(l',k',i')$, and \eqref{eq:relation:rho-delta} holds.

\smallskip

It will be significant that the probability under the sum in \eqref{eq:probability-tau:fixed-rho-delta} is nonzero only for those $\{z^{(k)}_{l,i}\}$, which satisfy certain constraints, similar to those satisfied by the points $P^{(k)}_{l,i}$. 

Recall the definition of affine hulls $H^{(k)}_{l,i}$ and $\check H^{(k)}_{l,i}$. The points $P^{(k)}_{l,i}$ satisfy:
\begin{itemize}
\item[(a)]
$\tfrac12\delta^{(k)}_{l,i}< d(0,H^{(k)}_{l,i})\leq \delta^{(k)}_{l,i}$ for all $(l,k,i)\prec(\ell,\kappa,\iota)$;
\item[(b)]
$d(P^{(k)}_{l,i},\check H^{(k)}_{l,i})> \varepsilon$ for all $(l,k,i)\prec(\ell,\kappa,\iota)$;
\item[(c)]
if $\sum\limits_{k=1}^K\big(\sum\limits_{l=0}^{\ell-1} n^{(k)}_l + \widetilde n^{(k)}_\ell\big)\neq d+1$ then either 
$d(P^{(\kappa)}_{\ell,\iota},\check H^{(\kappa)}_{\ell,\iota})<\varepsilon$ or 
$d(0,H^{(\kappa)}_{\ell,\iota})<2\gamma\varepsilon$. 
\end{itemize}
In analogy with $H^{(k)}_{l,i}$ and $\check H^{(k)}_{l,i}$, we define the affine hulls
\[
h^{(k)}_{l,i} = \mathrm{aff}\big\{z^{(k')}_{l',i'}\,:\,(l',k',i')\preceq(l,k,i)\big\},\quad \check h^{(k)}_{l,i} = \mathrm{aff}\{z^{(k')}_{l',i'},\,(l',k',i')\prec(l,k,i)\}.
\]

\begin{definition}\label{def:mathcal:z'}
Let $\mathcal Z'$ be the set of all tuples $\{z^{(k)}_{l,i}\in\mathcal Z,\,(l,k,i)\preceq(\ell,\kappa,\iota)\}$ such that 
\begin{itemize}
\item[(a)]
$\tfrac14\delta^{(k)}_{l,i}\leq d(0,h^{(k)}_{l,i})\leq 2\delta^{(k)}_{l,i}$ for all $(l,k,i)\prec(\ell,\kappa,\iota)$;
\item[(b)]
$d(z^{(k)}_{l,i},\check h^{(k)}_{l,i})> \tfrac12\varepsilon$ for all $(l,k,i)\prec(\ell,\kappa,\iota)$;
\item[(c)]
if $\sum\limits_{k=1}^K\big(\sum\limits_{l=0}^{\ell-1} n^{(k)}_l + \widetilde n^{(k)}_\ell\big)\neq d+1$ then either 
$d(z^{(\kappa)}_{\ell,\iota},\check h^{(\kappa)}_{\ell,\iota})<2\varepsilon$ or 
$d(0,h^{(\kappa)}_{\ell,\iota})<4\gamma\varepsilon$. 
\end{itemize} 
\end{definition}
By the definition of $\tau$ and Lemma~\ref{l:affine-hulls-properties}, if $\varepsilon<\varepsilon_*=\varepsilon_*(d)>0$, then 
the probability in \eqref{eq:probability-tau:fixed-rho-delta} is nonzero only for the tuples $\{z^{(k)}_{l,i}\}$ from $\mathcal Z'$. Thus, to complete the proof of Lemma~\ref{l:probability:tau}, it suffices to prove that 
\begin{equation}\label{eq:probability-tau:sum-over-mathcalz'}
\sum\limits_{\mathcal Z'}\mathsf P\Big[\begin{array}{l}F,G,N^{(k)}_l = n^{(k)}_l,\widetilde N^{(k)}_\ell = \widetilde n^{(k)}_\ell, r_l=\rho_l\\ P^{(k)}_{l,i}\in B^{(k)}_{l,i}\text{ for all }l,k,i\end{array}
\Big]\leq C\log^m(\tfrac1\varepsilon)\, \varepsilon^{d+1}.
\end{equation}

\medskip

With all the restrictions we have imposed so far, if $\ell\geq 1$, then there can still be some ambiguity in the order in which Brownian motions realize different excursions. 
The total number of ways points $P^{(k)}_{l,i}$ can be visited by $W^{(k)}$ in order is at most the number of permutations $\big(\sum\limits_{l=0}^{\ell-1}n^{(k)}_l + \widetilde n^{(k)}_\ell\big)!$, which is bounded from above by $(d+1)!$. Thus, the total number of possible cascades (for fixed $n^{(k)}_l$'s and $\widetilde n^{(k)}_\ell$'s) is at most $((d+1)!)^d$.

\smallskip

In fact, the order in which the points $P^{(k)}_{l,i}$ are visited by $W^{(k)}$ uniquely specifies the order in which $W^{(k)}$ realizes different excursions. (This is not entirely trivial, since $P^{(k)}_{l,i}$ is generally not a starting point of a $r_l$-excursion after $\ell$ iterations, but of a $r_{l'}$-excursion for some $l\leq l'\leq \ell$.)
To see this it is more convenient to work with the end-points of excursions. 

\smallskip

Recall the definition of $\overline\tau^{(k)}_{l,i}$ and $\overline\eta^{(k)}_{l,i}$ from \eqref{def:overline-tau-eta}. For $(l,k,i)\preceq (\ell,\kappa,\iota)$ with $l\geq 1$, let $j$ ($\geq 2$) be such that $P^{(k)}_{l,i} = \overline P^{(k)}_{l,j} = W^{(k)}_{\overline\tau^{(k)}_{l,j}}$. We denote the end-point of the $r_l$-excursion starting from $\overline P^{(k)}_{l,j-1}$ as 
\[
Q^{(k)}_{l,i} = W^{(k)}_{\overline\eta^{(k)}_{l,j-1}}\footnote{Note that after $\ell$ iterations $Q^{(k)}_{l,i}$ is still the end-point of a $r_l$-excursion, but generally with a different starting point $P^{(k)}_{l',i'}$ for $l'\geq l$.}
\] 
and, for a given tuple $\{z^{(k')}_{l',i'}\}$ from $\mathcal Z'$, consider the sets
\[
\mathbb Q^{(k)}_{l,i} = \big\{q\in \partial B(1+\rho_l)\,:\,
\|q-z^{(k')}_{l',i'}\|<2(\log\varepsilon)^2\rho_l\text{ for some }(l',k',i')\prec(l,k,i)\big\}.
\]
\begin{lemma}\label{l:properties-Qkli}
For every $(l,k,i)\preceq (\ell,\kappa,\iota)$ with $l\geq 1$, the following properties hold:
\begin{enumerate}
\item
after $Q^{(k)}_{l,i}$, the Brownian motion $W^{(k)}$ enters $B(1)$ through $P^{(k)}_{l,i}$, furthermore, if $l>1$, then it enters $B(1)$ before leaving $B(1+r_{l-1})$;
\item
if the event in \eqref{eq:probability-tau:sum-over-mathcalz'} occurs, then $Q^{(k)}_{l,i}\in \mathbb Q^{(k)}_{l,i}$.
\end{enumerate}
\end{lemma}
\begin{proof}
While the first statement in (1) follows from the definition of $Q^{(k)}_{l,i}$, the second statement follows from the definition of $P^{(k)}_{l,i}$, since it is a revisit point of a $r_l$-excursion but not of  a $r_{l-1}$-excursion.

\smallskip

To prove claim (2), let $Q^{(k)}_{l,i} = W^{(k)}_{\overline\eta^{(k)}_{l,j-1}}$. Note that $ \overline P^{(k)}_{l,j-1} = W^{(k)}_{\overline\tau^{(k)}_{l,j-1}} = P^{(k)}_{l',i'}$, for some $(l',k,i')\prec(l,k,i)$. By assumption, $\|P^{(k)}_{l',i'} - z^{(k)}_{l',i'}\|<\overline \varepsilon$. Thus, 
\[
\|Q^{(k)}_{l,i} - z^{(k)}_{l',i'}\|\leq 
\|Q^{(k)}_{l,i}- P^{(k)}_{l',i'}\| 
+
\|P^{(k)}_{l',i'}- z^{(k)}_{l',i'}\|
\leq (\log\varepsilon)^2\rho_l + \overline\varepsilon < 2(\log\varepsilon)^2\rho_l,
\]
where the third inequality holds by the definition of event $G$.
This proves claim (2).
\end{proof}

\smallskip

By Lemma~\ref{l:properties-Qkli} (1), the order in which the points $P^{(k)}_{l,i}$ are visited by $W^{(k)}$ uniquely specifies the order in which $W^{(k)}$ realizes different excursions. (Indeed, the excursion that ends in $Q^{(k)}_{l,i}$ is always a $r_l$-excursion.)
For $1\leq k\leq K$, let $a^{(k)}$ be an arbitrary fixed permutation of triples 
$\{(l',k',i')\preceq(\ell,\kappa,\iota)\,:\,l'\geq 1,\ k'=k\}$. 
(It is possible that the set is empty, e.g.\ when $\ell=0$ it is so for all $k\leq t$; if so, we write $a^{(k)}=\emptyset$.) 
We write $W^{(k)}\in a^{(k)}$ if the chronological order of the points $P^{(k)}_{l,i}$ is as in the vector $a^{(k)}$. (When $a^{(k)}=\emptyset$, it means that the Brownian motion $W^{(k)}$ does not contribute any excursion to the cascade, that is after leaving $B(1+\rho_\ell)$ it never returns to $B(1)$.)  
For any (admissible) vectors $a^{(1)},\ldots, a^{(K)}$, by repeatedly applying the strong Markov property at successive revisits of $\partial B(1)$ resp.\ exits from $B(1+\rho_l)$ (and by replacing the random locations of the Brownian motions at the stopping times by the suprema over the admissible ranges of locations) and using Lemma~\ref{l:properties-Qkli}, we obtain that for any $\ell\geq 1$,  
\begin{multline*}
\mathsf P\Big[\begin{array}{l}F,G,N^{(k)}_l = n^{(k)}_l,\widetilde N^{(k)}_\ell = \widetilde n^{(k)}_\ell, r_l=\rho_l\\ P^{(k)}_{l,i}\in B^{(k)}_{l,i}, W^{(k)}\in a^{(k)}\text{ for all }l,k,i\end{array}
\Big]\\
\begin{aligned}
\leq 
&\Big(\prod\limits_{k=1}^K \mathsf P[W_0^{(k)}\in B^{(k)}_{0,1}]\Big)\,\,\Big(\sup\limits_{p\in\partial B(1)}\mathsf P_p[T_{1-\varepsilon}=+\infty]\Big)^K\\
&\quad\prod\limits_{k=1}^K \Big(\prod\limits_{l=1}^{\ell-1}\big(\sup\limits_{p\in\partial B(1)}\mathsf P_p[T_{1+\rho_l}<T_{1-\varepsilon}]\big)^{n^{(k)}_l}\,\big(\sup\limits_{p\in\partial B(1)}\mathsf P_p[T_{1+\rho_\ell}<T_{1-\varepsilon}]\big)^{\widetilde n^{(k)}_\ell}\Big)\\
&\qquad\prod\limits_{\substack{(l,k,i)\preceq (\ell,\kappa,\iota)\\l\geq 1}}\sup\limits_{q\in \mathbb Q^{(k)}_{l,i}} \mathsf P_q[W_{T_1}\in B^{(k)}_{l,i}].
\end{aligned}
\end{multline*}
By Lemma~\ref{l:BM-hitting}, 
for $p\in\partial B(1)$ and $\varepsilon, \rho>0$, $\mathsf P_p[T_{1+\rho}<T_{1-\varepsilon}]\leq C\frac{\varepsilon}{\rho}$ and $\mathsf P_p[T_{1-\varepsilon}=+\infty]\leq C\varepsilon$. Thus, if we define 
\[
n_l = \sum\limits_{k=1}^K n^{(k)}_l\quad\text{resp.}\quad \widetilde n_\ell = \sum\limits_{k=1}^K \widetilde n^{(k)}_\ell \quad\big(\text{note that $K+\sum\limits_{l=1}^{\ell-1}n_l + \widetilde n_\ell = t$}\big),
\]
then the above probability is bounded from above by 
\begin{multline*}
\Big(\prod\limits_{l=1}^{\ell-1}\big(\frac{C\varepsilon}{\rho_l}\big)^{n_l}\Big)\,\big(\frac{C\varepsilon}{\rho_\ell}\big)^{\widetilde n_\ell}\,(C\varepsilon)^K\,
\Big(\prod\limits_{k=1}^K \mathsf P[W_0^{(k)}\in B^{(k)}_{0,1}]\Big)\,\,\prod\limits_{\substack{(l,k,i)\preceq (\ell,\kappa,\iota)\\l\geq 1}}\sup\limits_{q\in \mathbb Q^{(k)}_{l,i}} \mathsf P_q[W_{T_1}\in B^{(k)}_{l,i}]\\
=
(C\varepsilon)^t\,\big(\prod\limits_{l=1}^{\ell-1}\rho_l^{-n_l}\big)\,\rho_\ell^{-\widetilde n_\ell}\,\,\Big(\prod\limits_{k=1}^K \mathsf P[W_0^{(k)}\in B^{(k)}_{0,1}]\Big)\,\,\prod\limits_{\substack{(l,k,i)\preceq (\ell,\kappa,\iota)\\l\geq 1}}\sup\limits_{q\in \mathbb Q^{(k)}_{l,i}} \mathsf P_q[W_{T_1}\in B^{(k)}_{l,i}].
\end{multline*}

\smallskip

When $\ell=0$ (recall that this is equivalent to $t\leq K$), there is no need to specify the order of excursions and the inequality is much simplier: By Lemma~\ref{l:BM-hitting},
\begin{multline*}
\mathsf P\Big[F,G,\tau=t,P^{(k)}_{0,1}\in B^{(k)}_{0,1} \text{ for all }k\leq t\Big]\\
\leq \Big(\prod\limits_{k=1}^t \mathsf P[W_0^{(k)}\in B^{(k)}_{0,1}]\Big)\,\Big(\sup\limits_{p\in\partial B(1)}\mathsf P_p[T_{1-\varepsilon}=+\infty]\Big)^t
\leq (C\varepsilon)^t\,\,\Big(\prod\limits_{k=1}^{K\wedge t} \mathsf P[W_0^{(k)}\in B^{(k)}_{0,1}]\Big).
\end{multline*}

If we define 
\[
\rho_0=1, 
\]
then this bound has the same form as the bound obtained in the case $\ell\geq 1$. Thus, by summing over all admissible tuples $(a^{(1)},\ldots, a^{(K)})$, we obtain that the sum of probabilities in \eqref{eq:probability-tau:sum-over-mathcalz'} is bounded from above by 
\[
((d+1)!)^d\,(C\varepsilon)^t\,\big(\prod\limits_{l=1}^{\ell-1}\rho_l^{-n_l}\big)\,\rho_\ell^{-\widetilde n_\ell}\,
\sum\limits_{\mathcal Z'}\Big(\prod\limits_{k=1}^{K\wedge t} \mathsf P[W_0^{(k)}\in B^{(k)}_{0,1}]\,
\prod\limits_{\substack{(l,k,i)\preceq (\ell,\kappa,\iota)\\l\geq 1}}\sup\limits_{q\in \mathbb Q^{(k)}_{l,i}}\mathsf P_q[W_{T_1}\in B^{(k)}_{l,i}]\Big).
\]
Although it is not essential, we rewrite the probabilities involving the (uniformly distributed) starting points of the Brownian motions in terms of probabilities involving entrance points. This will allow us to estimate them in a unified way using Lemmas~\ref{l:affine-hull-hitting-1}, \ref{l:affine-hull-hitting-2} and \ref{l:affine-hull-hitting-3}.
In order to do this, we extend the definition of $\mathbb Q^{(k)}_{l,i}$ to $l=0$ as (recall that $\rho_0=1$)
\[
\mathbb Q^{(k)}_{0,1} = \big\{q\in \partial B(1+\rho_0)\,:\,
\|q-z^{(k')}_{0,1}\|<2(\log\varepsilon)^2\rho_0 \text{ for some }k'<k)\big\}, \quad 2\leq k\leq K\wedge t.
\]

By the Poisson formula (Lemma~\ref{l:poisson-formula}),
\[
\mathsf P[W_0^{(k)}\in B^{(k)}_{0,1}] = \frac{|B^{(k)}_{0,1}\cap\partial B(1)|}{|\partial B(1)|} \leq  C\sup\limits_{q\in \mathbb Q^{(k)}_{0,1}}\mathsf P_q[W_{T_1}\in B^{(k)}_{0,1}].
\]
This allows us to rewrite the last bound as 
\[
C\varepsilon^t\,\big(\prod\limits_{l=1}^{\ell-1}\rho_l^{-n_l}\big)\,\rho_\ell^{-\widetilde n_\ell}\,
\sum\limits_{\mathcal Z'}\Big(\mathsf P[W^{(1)}_0\in B^{(1)}_{0,1}]
\prod\limits_{\substack{(l,k,i)\preceq (\ell,\kappa,\iota)\\ (l,k,i)\neq (0,1,0)}}\sup\limits_{q\in \mathbb Q^{(k)}_{l,i}} \mathsf P_q[W_{T_1}\in B^{(k)}_{l,i}]\Big)
\]
for some $C=C(d)$.

\smallskip

Recall that there is a one-to-one correspondence between the triples $(l,k,i)\preceq(\ell,\kappa,\iota)$ (ordered lexicographically) and the numbers $\{1,\ldots, t\}$, given by the map $N^{(k)}_{l,i} = s$; cf.\ \eqref{def:Nkli}.
We relabel the tuples from $\mathcal Z'$ as $(z_1,\ldots, z_t)$, where $z_s$ stays for $z^{(k)}_{l,i}$ with $N^{(k)}_{l,i} = s$. We define $h_s = \mathrm{aff}\{z_1,\ldots, z_s\}$.
Similarly, we define $\delta_s = \delta^{(k)}_{l,i}$, when $N^{(k)}_{l,i} = s$. 
By the definition of $\mathcal Z'$, (a) $\tfrac14\delta_s\leq d(0,h_s)\leq 2\delta_s$ for all $s<t$ and (b) $d(z_s,h_{s-1})>\tfrac12\varepsilon$ for all $s<t$ and (c) if $t\neq d+1$, then either $d(z_t,h_{t-1})<2\varepsilon$ or $d(0,h_t)<4\gamma\varepsilon$. 
When $\ell=0$, we define $\overline\rho_s=1$ for $s\leq t$, and when $\ell\geq 1$, we define 
\[
\overline\rho_s = \left\{\begin{array}{ll}1 & s\leq K;\\[4pt] \rho_l & K+\sum\limits_{j=1}^{l-1}n_j< s\leq K+\sum\limits_{j=1}^l n_j\text{ for }1\leq l\leq \ell-1;\\[4pt] \rho_\ell & K+\sum\limits_{j=1}^{\ell-1} n_j <s \leq t.\end{array}\right.
\]
At last, for $s\geq 2$, let $\mathbb Q_s = \{q\in\partial B(1+\overline \rho_s)\,:\,\|q-z_{s'}\|<2(\log\varepsilon)^2\overline\rho_s\text{ for some }s'<s\}$.

\smallskip

With the new notation, we can rewrite the above bound as 
\begin{multline*}
C\varepsilon^t\,\Big(\prod\limits_{s=1}^t(\overline\rho_s)^{-1}\Big)
\sum\limits_{\mathcal Z'}\Big(\mathsf P[W^{(1)}_0\in B(z_1,\overline\varepsilon)]\,
\prod\limits_{s=2}^t \sup\limits_{q\in\mathbb Q_s}
\mathsf P_q[W_{T_1}\in B(z_s,\overline\varepsilon)]\Big)\\
=
C\varepsilon^t\,\Big(\prod\limits_{s=1}^t(\overline\rho_s)^{-1}\Big)
\sum\limits_{z_1\in \mathcal Z}\mathsf P[W^{(1)}_0\in B(z_1,\overline\varepsilon)]\,
\sum\limits_{z_2}\sup\limits_{q\in\mathbb Q_2} \mathsf P_q[W_{T_1}\in B(z_2,\overline\varepsilon)]\\
\ldots
\sum\limits_{z_t}\sup\limits_{q\in\mathbb Q_t} \mathsf P_q[W_{T_1}\in B(z_t,\overline\varepsilon)],
\end{multline*}
where $\sum\limits_{z_s}$ is the sum over (a) $z_s\in\mathcal Z$ such that $\tfrac14\delta_s\leq d(0,h_s)\leq 2\delta_s$ and $d(z_s,h_{s-1})>\tfrac12\varepsilon$ for given $(z_1,\ldots, z_{s-1})$, when $2\leq s<t$, (b) $z_t\in\mathcal Z$ for given $(z_1,\ldots, z_{t-1})$, when $t=d+1$, and (c) $z_t\in\mathcal Z$ such that either $d(z_t,h_{t-1})<2\varepsilon$ or $d(0,h_t)<4\gamma\varepsilon$ for given $(z_1,\ldots, z_{t-1})$, when $t<d+1$. Note that for $2\leq s\leq t$, 
\begin{multline*}
\sum\limits_{z_s}\sup\limits_{q\in\mathbb Q_s} \mathsf P_q[W_{T_1}\in B(z_s,\overline\varepsilon)]
\leq \sum\limits_{z_s}\sum\limits_{s'<s}\sup\limits_{\substack{q\in\partial B(1+\overline\rho_s)\\ \|q-z_{s'}\|\leq 2(\log\varepsilon)^2\overline\rho_s}} \mathsf P_q[W_{T_1}\in B(z_s,\overline\varepsilon)]\\
\begin{aligned}
\leq 
&\,\,d\,\max\limits_{s'<s}\Big(\sum\limits_{z_s}\sup\limits_{\substack{q\in\partial B(1+\overline\rho_s)\\\|q-z_{s'}\|\leq 2(\log\varepsilon)^2\overline\rho_s}} \mathsf P_q[W_{T_1}\in B(z_s,\overline\varepsilon)]\Big)\\
\leq
&\,\,d \sup\limits_{h\in \mathcal H_{s-1}}\sup\limits_{y\in h\cap\partial B(1)}
\Big(\sum\limits_{z\in\mathcal Z_{h,s}}
\sup\limits_{\substack{q\in\partial B(1+\overline\rho_s)\\\|q-y\|\leq 2(\log\varepsilon)^2\overline\rho_s}} \mathsf P_q[W_{T_1}\in B(z,\overline\varepsilon)]\Big),
\end{aligned}
\end{multline*}
where for $s'<t$, $\mathcal H_{s'}$ is the set of $(s'-1)$-dimensional affine spaces $h$ intersecting $\partial B(1)$ such that 
$\tfrac14\delta_{s'}\leq d(0,h)\leq 2\delta_{s'}$ (note that $h_{s-1}\in\mathcal H_{s-1}$ for every $2\leq s\leq t$, by conditions (a) and (b) from the definition of $\mathcal Z'$), $\mathcal Z_{h,s}=\{z\in\mathcal Z\,:\,d(0,\mathrm{aff}\{z,h\})\leq 2\delta_s\}$ for $s<t$, $\mathcal Z_{h,t} = \mathcal Z$ when $t=d+1$, and $\mathcal Z_{h,t} = \{z\in\mathcal Z\,:\,d(z,h)<2\varepsilon\text{ or }d(0,\mathrm{aff}\{z,h\})<4\gamma\varepsilon\}$ when $t<d+1$. Note that the final estimate does not depend on $z_1,\ldots, z_{s-1}$ anymore. 
Furthermore, $\sum\limits_{z_1\in \mathcal Z}\mathsf P[W^{(1)}_0\in B(z_1,\overline\varepsilon)]\leq 3^d$ by the definition of $\mathcal Z$ (see  \eqref{def:mathcal:Z}). All in all, we obtain that the sum of probabilities in \eqref{eq:probability-tau:sum-over-mathcalz'} is bounded from above by 
\begin{equation}\label{eq:probability-tau:final}
C\varepsilon^t\,\Big(\prod\limits_{s=1}^t(\overline\rho_s)^{-1}\Big)\,
\prod\limits_{s=2}^t\sup\limits_{h\in \mathcal H_{s-1}}\sup\limits_{y\in h\cap\partial B(1)}\Big(\sum\limits_{z\in\mathcal Z_{h,s}}\sup\limits_{\substack{q\in\partial B(1+\overline\rho_s)\\\|q-y\|\leq 2(\log\varepsilon)^2\overline\rho_s}} \mathsf P_q[W_{T_1}\in B(z,\overline\varepsilon)]\Big)
\end{equation}
for some $C=C(d)$. Thus, it suffices to prove that \eqref{eq:probability-tau:final} is bounded from above by $C\log^m\big(\tfrac1\varepsilon\big)\varepsilon^{d+1}$ for some $C=C(d)$ and $m=m(d)$. We estimate the sums over $\mathcal Z_{h,s}$ in \eqref{eq:probability-tau:final} using Lemmas~\ref{l:affine-hull-hitting-1}, \ref{l:affine-hull-hitting-2} and \ref{l:affine-hull-hitting-3}.

By Lemma~\ref{l:affine-hull-hitting-1} (and recalling that $\delta_1=1$),
\begin{multline*}
\prod\limits_{s=2}^{t-1}\sup\limits_{h\in \mathcal H_{s-1}}\sup\limits_{y\in h\cap\partial B(1)}\Big(\sum\limits_{z\in\mathcal Z_{h,s}}\sup\limits_{\substack{q\in\partial B(1+\overline\rho_s)\\\|q-y\|\leq 2(\log\varepsilon)^2\overline\rho_s}} \mathsf P_q[W_{T_1}\in B(z,\overline\varepsilon)]\Big)\\
\leq 
\prod\limits_{s=2}^{t-1}\Big(C(\log\varepsilon)^{2(d-1)}\big(\tfrac {\delta_s}{\delta_{s-1}}\big)^{d-1-(s-2)}\Big)
\leq 
C^d(\log\varepsilon)^{2d^2}\Big(\prod\limits_{s=1}^{t-1}\delta_s\Big)\,\big(\delta_{t-1}\big)^{d+1-t}.
\end{multline*}
By \eqref{eq:relation:rho-delta}, $\rho_l\geq \tfrac14\gamma^{-1}\delta_s$ for all $s\geq K+ \sum\limits_{j=1}^{l-1}n_j$. Thus, $\overline\rho_s\geq \tfrac14\gamma^{-1}\delta_{s-1}$ for all $s$. Recalling that $\overline\rho_1=1$, the above expression is at most 
\[
C^d(\log\varepsilon)^{2d^2}(4\gamma)^d\Big(\prod\limits_{s=1}^t\overline\rho_s\Big)\,\big(\delta_{t-1}\big)^{d+1-t} \stackrel{\eqref{def:gamma}}\leq 
(8C)^d(\log\varepsilon)^{2d(d+1)}\Big(\prod\limits_{s=1}^t\overline\rho_s\Big)\,\big(\delta_{t-1}\big)^{d+1-t}.
\]
Thus, to prove that $\eqref{eq:probability-tau:final}\leq C\log^m\big(\tfrac1\varepsilon\big)\varepsilon^{d+1}$, if suffices to show that 
\begin{equation}\label{eq:probability-tau:final:2}
\sup\limits_{h\in \mathcal H_{t-1}}\sup\limits_{y\in h\cap\partial B(1)}\Big(\sum\limits_{z\in\mathcal Z_{h,t}}\sup\limits_{\substack{q\in\partial B(1+\overline\rho_t)\\\|q-y\|\leq 2(\log\varepsilon)^2\overline\rho_t}} \mathsf P_q[W_{T_1}\in B(z,\overline\varepsilon)]\Big)
\leq C\log^m\big(\tfrac1\varepsilon\big)\big(\tfrac{\varepsilon}{\delta_{t-1}}\big)^{d+1-t}
\end{equation}
for some $C=C(d)$ and $m=m(d)$.
We consider separately three cases, which correspond to three stopping rules in the definition of $\tau$ (Definition~\ref{def:tau}) resp.\ the three conditions in the definition of $\mathcal Z_{h,t}$.

If $t=d+1$, then by Lemma~\ref{l:affine-hull-hitting-3}, 
\begin{multline*}
\sup\limits_{h\in \mathcal H_{t-1}}\sup\limits_{y\in h\cap\partial B(1)}\Big(\sum\limits_{z\in\mathcal Z_{h,t}}\sup\limits_{\substack{q\in\partial B(1+\overline\rho_t)\\\|q-y\|\leq 2(\log\varepsilon)^2\overline\rho_t}} \mathsf P_q[W_{T_1}\in B(z,\overline\varepsilon)]\Big)\\
\begin{aligned}
= &\,\,\sup\limits_{y\in \partial B(1)}\Big(\sum\limits_{z\in\mathcal Z}\sup\limits_{\substack{q\in\partial B(1+\overline\rho_t)\\\|q-y\|\leq 2(\log\varepsilon)^2\overline\rho_t}} \mathsf P_q[W_{T_1}\in B(z,\overline\varepsilon)]\Big)\\
\leq &\,\,C(\log\varepsilon)^{2(d-1)} \stackrel{(t=d+1)}= C(\log\varepsilon)^{2(d-1)} \big(\tfrac{\varepsilon}{\delta_{t-1}}\big)^{d+1-t},
\end{aligned}
\end{multline*}
hence \eqref{eq:probability-tau:final:2} holds e.g.\ with $m=2(d-1)$.

\smallskip

If $t\neq d+1$, then by Lemma~\ref{l:affine-hull-hitting-1}, 
\begin{multline*}
\sup\limits_{h\in \mathcal H_{t-1}}\sup\limits_{y\in h\cap\partial B(1)}\Big(\sum\limits_{z\in\mathcal Z\,:\,d(0,\mathrm{aff}\{z,h\})\leq 4\gamma\varepsilon}\sup\limits_{\substack{q\in\partial B(1+\overline\rho_t)\\\|q-y\|\leq 2(\log\varepsilon)^2\overline\rho_t}} \mathsf P_q[W_{T_1}\in B(z,\overline\varepsilon)]\Big)\\
\leq 
C(\log\varepsilon)^{2(d-1)}\big(\tfrac {4\gamma\varepsilon}{\delta_{t-1}}\big)^{d-1-(t-2)}\stackrel{\eqref{def:gamma}}\leq
8^dC(\log\varepsilon)^{4(d-1)}\big(\tfrac {\varepsilon}{\delta_{t-1}}\big)^{d+1-t}
\end{multline*}
and by Lemma~\ref{l:affine-hull-hitting-2}, 
\begin{multline*}
\sup\limits_{h\in \mathcal H_{t-1}}\sup\limits_{y\in h\cap\partial B(1)}\Big(\sum\limits_{z\in\mathcal Z\,:\,d(z,h)\leq 2\varepsilon}\sup\limits_{\substack{q\in\partial B(1+\overline\rho_t)\\\|q-y\|\leq 2(\log\varepsilon)^2\overline\rho_t}} \mathsf P_q[W_{T_1}\in B(z,\overline\varepsilon)]\Big)\\
\begin{aligned}
\leq 
C(\log\varepsilon)^{2(t-3)^+}\big(\tfrac{2\varepsilon}{\overline\rho_t}\big)^{(d-1)-(t-3)^+}
\stackrel{(*)}\leq 
&\,\,C(\log\varepsilon)^{2d}\,\big(\tfrac{8\gamma\varepsilon}{\delta_{t-1}}\big)^{d+1-t}\,
\big(\tfrac{2\varepsilon}{\overline\rho_t}\big)^{(t-2)-(t-3)^+}\\
\stackrel{(**)}\leq 
&\,\,32^dC(\log\varepsilon)^{4d}\,\big(\tfrac {\varepsilon}{\delta_{t-1}}\big)^{d+1-t},
\end{aligned}
\end{multline*}
where $(*)$ follows from $\overline\rho_t\geq \tfrac14\gamma^{-1}\delta_{t-1}$ and 
$(**)$ from \eqref{def:gamma} and $\overline\rho_t\geq \varepsilon$ (and $t\geq 2$). Hence, \eqref{eq:probability-tau:final:2} holds also in the case $t\neq d+1$ e.g.\ with $m=4d$. 

\smallskip

The proof of \eqref{eq:probability-tau:final:2} and hence of Lemma~\ref{l:probability:tau} is completed. \qed

\subsection{Small perturbations of affine hulls}\label{sec:aux-affine-hulls}

In this section we collect some deterministic results on small perturbations of affine hulls, which are essential for the reduction to \eqref{eq:probability-tau:sum-over-mathcalz'} in the proof of Lemma~\ref{l:probability:tau}. We begin with a general result in Proposition~\ref{prop:distance-affine-hull} and collect specific applications that we need for \eqref{eq:probability-tau:sum-over-mathcalz'} in Lemma~\ref{l:affine-hulls-properties}.

\smallskip

The \emph{affine hull} of a set $S\subseteq \R^d$, denoted by $\mathrm{aff}\{S\}$, is the smallest affine subspace of $\R^d$ containing $S$.\footnote{The affine hull of a set of points $x_0,\ldots, x_n\in\R^d$ can be defined as
\[
\mathrm{aff}\{x_0,\ldots, x_n\} = \Big\{\sum\limits_{i=0}^n\alpha_ix_i\,:\,\alpha_0,\ldots, \alpha_n\in\R^d,\,\sum\limits_{i=0}^n\alpha_i=1\Big\}
\]}
\begin{proposition}\label{prop:distance-affine-hull}
Let $R<\infty$ and $p>0$. There exists $C=C(d,R,p)$, such that for all $\delta\in(0,1)$, $n\in\{1,\ldots, d\}$, $x_0,\ldots, x_n\in\R^d$ and $y_0,\ldots, y_n\in\R^d$, if 
\begin{itemize}
 \item[(a)]
 $\|x_i\|\leq R$ for all $0\leq i\leq n$; 
 \item[(b)]
 $d(x_i,\mathrm{aff}\{x_0,\ldots, x_{i-1}\})>\delta$ for all $1\leq i\leq n$;
 \item[(c)]
 $\|x_i-y_i\|<R\delta^{2d+p}$ for all $0\leq i\leq n$,
 \end{itemize}
then 
\begin{equation}\label{eq:distance-affine-hull-xy}
\big|d(x_i,\mathrm{aff}\{x_0,\ldots, x_{i-1}\}) - d(y_i,\mathrm{aff}\{y_0,\ldots, y_{i-1}\})\big|
\leq C\delta^p
\end{equation}
for all $1\leq i\leq n$.
\end{proposition}
\begin{proof}
It suffices to prove the result for $i=n$ and $\delta\leq \delta_0$ for $\delta_0 = \delta_0(d,R,p)$ small enough.

Without loss of generality we may assume that $x_0=y_0=0$. (The general result follows by considering translated points $\widetilde x_i = x_i-x_0$ and $\widetilde y_j = y_j-y_0$.) In this case, $V_i:=\mathrm{aff}\{0,x_1,\ldots,x_i\}$ and $W_j:=\mathrm{aff}\{0,y_1,\ldots, y_j\}$ are vector subspaces of $\R^d$; hence
\begin{equation}\label{eq:Gram-distance}
d(x_k,V_{k-1})^2 = \frac{G(x_1,\ldots,x_{k-1},x_k)}{G(x_1,\ldots,x_{k-1})},
\end{equation}
where 
\[
G(x_1,\ldots, x_k) = 
\mathrm{det}\left( \begin{array}{cccc}
    \langle x_1,x_1\rangle & \langle x_1,x_2\rangle & \cdots & \langle x_1,x_k\rangle\\
    \langle x_2,x_1\rangle & \langle x_2,x_2\rangle & \cdots & \langle x_2,x_k\rangle\\
    \vdots & \vdots & \ddots & \vdots\\
    \langle x_k,x_1\rangle & \langle x_k,x_2\rangle & \cdots & \langle x_k,x_k\rangle
  \end{array}  \right)
\]
is the Gram determinant. Similar formula holds for $d(y_k,W_{k-1})$.

\smallskip

Consider $a_{ij} = \langle x_i,x_j\rangle$, $b_{ij} = \langle y_i,y_j\rangle$ and $\epsilon_{ij} = b_{ij}-a_{ij}$. By the Cauchy-Schwarz inequality and assumptions (a) and (c), 
\[
|\varepsilon_{ij}|\leq \|x_i\|\,\|x_j-y_j\| + \|x_j\|\,\|x_i-y_i\| + \|x_i-y_i\|\,\|x_j-y_j\| \leq 2R^2\delta^{2d+p} + R^2\delta^{2(2d+p)}.
\]
Let $C_*=C_*(R)\geq1$ be such that $|a_{ij}|\leq C_*$, $|b_{ij}|\leq C_*$ and $|\varepsilon_{ij}|\leq C_*\delta^{2d+p}$.
We have 
\begin{eqnarray*}
G(y_1,\ldots, y_k) &= 
&\mathrm{det}\left( \begin{array}{ccc}
    b_{11}&  \cdots & b_{1k}\\
    \vdots &  \ddots & \vdots\\
    b_{k1} &  \cdots & b_{kk}
  \end{array}  \right)
= \sum\limits_{\sigma\in S_k}(-1)^{\mathrm{sign}(\sigma)}\,b_{1\sigma(1)}\ldots b_{k\sigma(k)}\\
&= &\sum\limits_{\sigma\in S_k}(-1)^{\mathrm{sign}(\sigma)}\,a_{1\sigma(1)}\ldots a_{k\sigma(k)} + R_k\\
&= &G(x_1,\ldots,x_k) + R_k,
\end{eqnarray*}
where
\begin{eqnarray*}
|R_k| &= 
&\big|
\sum\limits_{\sigma\in S_k}(-1)^{\mathrm{sign}(\sigma)}\,b_{1\sigma(1)}\ldots b_{k\sigma(k)}
- \sum\limits_{\sigma\in S_k}(-1)^{\mathrm{sign}(\sigma)}\,a_{1\sigma(1)}\ldots a_{k\sigma(k)}
\big|\\
&\leq &\sum\limits_{\sigma\in S_k}\big|
b_{1\sigma(1)}\ldots b_{k\sigma(k)} - a_{1\sigma(1)}\ldots a_{k\sigma(k)}
\big| 
\leq \sum\limits_{\sigma\in S_k} k C_*^k\delta^{2d+p} = k!kC_*^k\delta^{2d+p}\\
&\leq &d!dC_*^d\delta^{2d+p}.
\end{eqnarray*}
Furthermore, by assumption (b), 
\[
G(x_1,\ldots,x_k) = d(x_1,V_0)^2\,d(x_2,V_1)^2\ldots d(x_k,V_{k-1})^2 > \delta^{2k}\geq \delta^{2d}.
\]
Thus, $|R_k|\leq d!dC_*^d G(x_1,\ldots, x_k)\delta^p$ and we obtain for each $k\leq n$,
\[
(1-d!dC_*^d\delta^p)G(x_1,\ldots, x_k)\leq G(y_1,\ldots, y_k)\leq (1+d!dC_*^d\delta^p)G(x_1,\ldots,x_k).
\]
Plugging this into \eqref{eq:Gram-distance} gives that there is a constant $C_{**}=C_{**}(d,R,p)$, such that for all $k$ and $\delta\leq \delta_0(d,R,p)$,
\[
(1-C_{**}\delta^p)d(x_k,V_{k-1})\leq d(y_k,W_{k-1})\leq (1+C_{**}\delta^p)d(x_k,V_{k-1}).
\]
Finally, $d(x_k,V_{k-1})\leq 2R$ by assumption (a), thus we obtain \eqref{eq:distance-affine-hull-xy} with $C=2RC_{**}$.
\end{proof}

\medskip

In the following lemma, we collect all the applications of Proposition~\ref{prop:distance-affine-hull} used in the proof of Lemma~\ref{l:probability:tau}; cf.\ Definition~\ref{def:mathcal:z'} and below.

\begin{lemma}\label{l:affine-hulls-properties}
Let $x_0,\ldots, x_n\in\R^d$ and $y_0,\ldots, y_n\in\R^d$ satisfy the assumptions of Proposition~\ref{prop:distance-affine-hull} with $p>1$. Then there exists $\delta_0 = \delta_0(d,R,p)>0$ such that the following statements hold for all $\delta\in(0,\delta_0)$, $\delta_1\in[\delta,1)$ and $i\in\{1,\ldots, n\}$.
\begin{enumerate}
\item
If $d(x_i,\mathrm{aff}\{x_0,\ldots, x_{i-1}\}) > \delta_1$, then $d(y_i,\mathrm{aff}\{y_0,\ldots, y_{i-1}\})> \tfrac12\delta_1$.

\item
If $\delta_1 < d(0,\mathrm{aff}\{x_0,\ldots, x_{i-1}\}) \leq 2\delta_1$, then $\tfrac12\delta_1 <d(0,\mathrm{aff}\{y_0,\ldots, y_{i-1}\})< 4\delta_1$.

\item
If $d(x,\mathrm{aff}\{x_0,\ldots, x_{i-1}\})<\delta_1$ for some $x\in\R^d$ with $\|x\|\leq R$, then 

$d(y,\mathrm{aff}\{y_0,\ldots, y_{i-1}\})<2\delta_1$ for any $y\in\R^d$ with $\|y-x\|\leq R\delta^{2d+p}$.
\end{enumerate}
\end{lemma}
\begin{proof}
By Proposition~\ref{prop:distance-affine-hull}, 
\[
\big|d(x_i,\mathrm{aff}\{x_0,\ldots, x_{i-1}\}) - d(y_i,\mathrm{aff}\{y_0,\ldots, y_{i-1}\})\big|\leq C\delta^p<\tfrac12\delta\leq \tfrac12\delta_1,
\]
for all $\delta$ small enough. 
The first statement now follows from the assumption on $d(x_i,\mathrm{aff}\{x_0,\ldots, x_{i-1}\})$.

\smallskip

Since $\delta_1\in[\delta,1)$, the points $\{x_0,\ldots, x_{i-1},0\}$ and $\{y_0,\ldots, y_{i-1},0\}$ satisfy the assumptions of Proposition~\ref{prop:distance-affine-hull}. Thus, 
\[
\big|d(0,\mathrm{aff}\{x_0,\ldots, x_{i-1}\}) - d(0,\mathrm{aff}\{y_0,\ldots, y_{i-1}\})\big|\leq C\delta^p<\tfrac12\delta\leq \tfrac12\delta_1,
\]
for all $\delta$ small enough. The second statement now follows from the assumption on $d(0,\mathrm{aff}\{x_0,\ldots, x_{i-1}\})$.

\smallskip

We prove the third statement by contradiction. Assume $d(y,\mathrm{aff}\{y_0,\ldots, y_{i-1}\})\geq 2\delta_1$. By the first statement, the points $\{y_0,\ldots, y_{i-1},y\}$ satisfy the assumptions (on $x_0,\ldots, x_i$) of Proposition~\ref{prop:distance-affine-hull} with $\delta:=\tfrac12\delta$ (and suitably enlarged $R$). Thus, 
\[
\big|d(y,\mathrm{aff}\{y_0,\ldots, y_{i-1}\})-d(x,\mathrm{aff}\{x_0,\ldots, x_{i-1}\})\big|\leq C'\delta^p<\tfrac12\delta\leq \tfrac12\delta_1,
\]
for all $\delta$ small enough; hence 
\[
d(x,\mathrm{aff}\{x_0,\ldots, x_{i-1}\}) >d(y,\mathrm{aff}\{y_0,\ldots, y_{i-1}\}) - \tfrac12\delta_1 >\delta_1, 
\]
which contradicts the assumption $d(x,\mathrm{aff}\{x_0,\ldots, x_{i-1}\})<\delta_1$. Claim (3) is proven. 
\end{proof}

\subsection{Some results on hitting probabilities for Brownian motion related to affine spaces}\label{sec:aux-BM-affine-hulls}

In this section we discuss three applications of the Poisson formula (see Lemma~\ref{l:poisson-formula}), which we use to obtain final estimates in the proof of Lemma~\ref{l:probability:tau}; see the proof of \eqref{eq:probability-tau:final} and \eqref{eq:probability-tau:final:2}.

\begin{lemma}\label{l:affine-hull-hitting-1}
Let $d\geq 3$. Let $\delta\in(0,\tfrac12]$, $\rho\in(0,1]$, $k\in\{0,\ldots, d-1\}$ and $D>1$. 
Let $h$ be an affine space of dimension $k$ such that $\delta\leq d(0,h)\leq 2\delta$. Let $y\in h\cap\partial B(1)$.
Let $\overline\varepsilon\in(0,1)$ and let $\mathcal Z\subseteq\partial B(1)$ be such that for any $x\in\partial B(1)$, $1\leq \big|\{z\in\mathcal Z\,:\,\|z-x\|<\overline\varepsilon\}\big|\leq 3^d$. Then there exists $C=C(d)$ such that for $\delta'\in(0,\delta]$,
\[
\sum\limits_{z\in\mathcal Z:d(0,\mathrm{aff}\{h,z\})\leq \delta'}
\sup\limits_{\substack{q\in\partial B(1+\rho)\\ \|q-y\|\leq D\rho}}
\mathsf P_q[W_{T_1}\in B(z,\overline\varepsilon)]\leq CD^{d-1}\big(\tfrac {\delta'}{\delta}\big)^{d-1-k}.
\]
\end{lemma}
\begin{proof}[Proof of Lemma~\ref{l:affine-hull-hitting-1}]
We will apply the Poisson formula (Lemma~\ref{l:poisson-formula}) with $r=1$, $x=q$ and $B=B(z,\overline\varepsilon)\cap\partial B(1)$.

Let $\mathcal Z_0$ be the set of all $z\in\mathcal Z$ such that $d(0,\mathrm{aff}\{h,z\})\leq \delta'$ and $\|y-z\|\leq 2D\rho$ and let $\mathcal Z_i$ be the set of all $z\in\mathcal Z$ such that $d(0,\mathrm{aff}\{h,z\})\leq \delta'$ and $2^iD\rho\leq \|y-z\|\leq 2^{i+1}D\rho$, for $i\geq 1$. 
Note that 
\[
|\mathcal Z_i|\leq C \big(\tfrac{1}{\overline\varepsilon}\big)^{d-1}\,(2^iD\rho)^k\,\big(\tfrac{\delta'}{\delta}2^iD\rho\big)^{d-1-k} = C\big(\tfrac{2^iD\rho}{\overline\varepsilon}\big)^{d-1}\,\big(\tfrac{\delta'}{\delta}\big)^{d-1-k}
\]
and for each $z\in\mathcal Z_i$ and $q\in\partial B(1+\rho)$ with $\|q-y\|\leq D\rho$, we have $\|q-z\|\geq \|y-z\|-\|y-q\|\geq 2^{i-1}D\rho$ (when $i\geq 1$) and $\|q-z\|\geq \rho$ (when $i=0$); thus by the Poisson formula, 
\begin{equation}\label{eq:poisson-formula-zi}
\mathsf P_q[W_{T_1}\in B(z,\overline\varepsilon)] \leq 
C\left\{\begin{array}{ll}
\tfrac{\rho\,\overline\varepsilon^{d-1}}{\rho^d} = 
\big(\tfrac{\overline\varepsilon}{\rho}\big)^{d-1} & \text{for }i=0;\\[4pt]
\tfrac{\rho\,\overline\varepsilon^{d-1}}{(2^iD\rho)^d} 
\leq \tfrac{1}{2^{id}}\big(\tfrac{\overline\varepsilon}{\rho}\big)^{d-1}
&\text{for }i\geq 1.\end{array}\right.
\end{equation}
As a result, 
\begin{multline*}
\sum\limits_{z\in\mathcal Z:d(0,\mathrm{aff}\{h,z\})\leq \delta'}
\sup\limits_{\substack{q\in\partial B(1+\rho)\\ \|q-y\|\leq D\rho}}
\mathsf P_q[W_{T_1}\in B(z,\overline\varepsilon)] 
\leq \sum_{i=0}^\infty\sum\limits_{z\in\mathcal Z_i}
\sup\limits_{\substack{q\in\partial B(1+\rho)\\ \|q-y\|\leq D\rho}}
\mathsf P_q[W_{T_1}\in B(z,\overline\varepsilon)]\\
\begin{aligned}
\leq 
&\,\,\sum\limits_{i=0}^\infty\Big(C\big(\tfrac{2^iD\rho}{\overline\varepsilon}\big)^{d-1}\,\big(\tfrac{\delta'}{\delta}\big)^{d-1-k}\Big)\,\Big(C\tfrac{1}{2^{id}}\big(\tfrac{\overline\varepsilon}{\rho}\big)^{d-1}\Big)\\
= &\,\,2C^2D^{d-1}\big(\tfrac{\delta'}{\delta}\big)^{d-1-k}.
\end{aligned}
\end{multline*}
The proof is completed.
\end{proof}

\smallskip

\begin{lemma}\label{l:affine-hull-hitting-2}
Let $d\geq 3$. Let $\rho\in(0,1]$, $k\in\{0,\ldots, d-1\}$ and $D>1$. Let $h$ be an affine space of dimension $k$ such that $h\cap\partial B(1)\neq\emptyset$. Let $y\in h\cap\partial B(1)$. 
Let $\overline\varepsilon\in(0,1)$ and let $\mathcal Z\subseteq\partial B(1)$ be such that for any $x\in\partial B(1)$, $1\leq \big|\{z\in\mathcal Z\,:\,\|z-x\|<\overline\varepsilon\}\big|\leq 3^d$. Then there exists $C=C(d)$ such that for $\delta'\in(0,1)$,
\[
\sum\limits_{z\in\mathcal Z:d(z,h)\leq \delta'}
\sup\limits_{\substack{q\in\partial B(1+\rho)\\ \|q-y\|\leq D\rho}}
\mathsf P_q[W_{T_1}\in B(z,\overline\varepsilon)]\leq CD^{(k-1)^+}\big(\tfrac{\delta'}{\rho}\big)^{(d-1)-(k-1)^+}.
\]
\end{lemma}

\begin{proof}[Proof of Lemma~\ref{l:affine-hull-hitting-2}]
The argument is similar to the proof of Lemma~\ref{l:affine-hull-hitting-1}.
Let $\mathcal Z_0$ be the set of all $z\in\mathcal Z$ such that $d(z,h)\leq \delta'$ and $\|z-y\|\leq 2D\rho$ and let $\mathcal Z_i$ be the set of all $z\in\mathcal Z$ such that $d(z,h)\leq \delta'$ and $2^iD\rho\leq \|z-y\|\leq 2^{i+1}D\rho$, for $i\geq 1$. 
Note that 
\[
|\mathcal Z_i|\leq C\big(\tfrac{1}{\overline\varepsilon}\big)^{d-1}\,(2^iD\rho)^{(k-1)^+}\,(\delta')^{(d-1)-(k-1)^+}
\]
and for each $z\in\mathcal Z_i$ and $q\in\partial B(1+\rho)$ with $\|q-y\|\leq D\rho$, 
the probability $P_q[W_{T_1}\in B(z,\overline\varepsilon)]$ 
is bounded exactly as in \eqref{eq:poisson-formula-zi}. Thus,
\begin{multline*}
\sum\limits_{z\in\mathcal Z:d(z,h)\leq \delta'}
\sup\limits_{\substack{q\in\partial B(1+\rho)\\ \|q-y\|\leq D\rho}}
\mathsf P_q[W_{T_1}\in B(z,\overline\varepsilon)]
\leq \sum_{i=0}^\infty\sum\limits_{z\in\mathcal Z_i}
\sup\limits_{\substack{q\in\partial B(1+\rho)\\ \|q-y\|\leq D\rho}}
\mathsf P_q[W_{T_1}\in B(z,\overline\varepsilon)]\\
\begin{aligned}
\leq 
&\,\,\sum\limits_{i=0}^\infty\Big(C\big(\tfrac{1}{\overline\varepsilon}\big)^{d-1}\,(2^iD\rho)^{(k-1)^+}\,(\delta')^{(d-1)-(k-1)^+}\Big)\,\Big(C\tfrac{1}{2^{id}}\big(\tfrac{\overline\varepsilon}{\rho}\big)^{d-1}\Big)\\
\leq &\,\,2C^2D^{(k-1)^+}\big(\tfrac{\delta'}{\rho}\big)^{(d-1)-(k-1)^+}.
\end{aligned}
\end{multline*}
The proof is completed. 
\end{proof}

\smallskip

\begin{lemma}\label{l:affine-hull-hitting-3}
Let $d\geq 3$. Let $\rho\in(0,1]$ and $D>1$. Let $y\in \partial B(1)$. 
Let $\overline\varepsilon\in(0,1)$ and let $\mathcal Z\subseteq\partial B(1)$ be such that for any $x\in\partial B(1)$, $1\leq \big|\{z\in\mathcal Z\,:\,\|z-x\|<\overline\varepsilon\}\big|\leq 3^d$. Then there exists $C=C(d)$ such that
\[
\sum\limits_{z\in\mathcal Z}
\sup\limits_{\substack{q\in\partial B(1+\rho)\\ \|q-y\|\leq D\rho}}
\mathsf P_q[W_{T_1}\in B(z,\overline\varepsilon)]\leq C D^{d-1}.
\]
\end{lemma}
\begin{proof}
The argument is similar to the proof of Lemmas~\ref{l:affine-hull-hitting-1} and \ref{l:affine-hull-hitting-2}. 
Let $\mathcal Z_0$ be the set of all $z\in\mathcal Z$ such that $\|z-y\|\leq 2D\rho$ and let $\mathcal Z_i$ be the set of all $z\in\mathcal Z$ such that $2^iD\rho\leq \|z-y\|\leq 2^{i+1}D\rho$, for $i\geq 1$. 
Note that 
\[
|\mathcal Z_i|\leq C\big(\tfrac{1}{\overline\varepsilon}\big)^{d-1}\,(2^iD\rho)^{d-1}
\]
and for each $z\in\mathcal Z_i$ and $q\in\partial B(1+\rho)$ with $\|q-y\|\leq D\rho$, 
the probability $P_q[W_{T_1}\in B(z,\overline\varepsilon)]$ 
is bounded exactly as in \eqref{eq:poisson-formula-zi}. Thus,
\begin{eqnarray*}
\sum\limits_{z\in\mathcal Z}
\sup\limits_{\substack{q\in\partial B(1+\rho)\\ \|q-y\|\leq D\rho}}
\mathsf P_q[W_{T_1}\in B(z,\overline\varepsilon)]
&\leq &\sum_{i=0}^\infty\sum\limits_{z\in\mathcal Z_i}
\sup\limits_{\substack{q\in\partial B(1+\rho)\\ \|q-y\|\leq D\rho}}
\mathsf P_q[W_{T_1}\in B(z,\overline\varepsilon)]\\
&\leq 
&\sum\limits_{i=0}^\infty\Big(C\big(\tfrac{1}{\overline\varepsilon}\big)^{d-1}\,(2^iD\rho)^{d-1}\Big)\,\Big(C\tfrac{1}{2^{id}}\big(\tfrac{\overline\varepsilon}{\rho}\big)^{d-1}\Big)\\
&\leq &2C^2D^{d-1}.
\end{eqnarray*}
The proof is completed. 
\end{proof}

\section{Microscopic uniqueness for the vacant set of several Wiener sausages}\label{sec:uniqueness-WS}

Let $W^{(1)},W^{(2)},\ldots$ be independent Brownian motions in $\R^d$ ($d\geq 3$). For $y_1,\ldots, y_K\in\R^d$, we write $\mathsf P_{y_1,\ldots, y_K}$ for the law of $W^{(1)},\ldots, W^{(K)}$ when $W^{(1)}_0=y_1,\ldots, W^{(K)}_0 = y_K$. Let
\[
\mathcal V_K(W^{(1)},\ldots, W^{(K)}) = \R^d\setminus\Big(\bigcup\limits_{k=1}^K\bigcup\limits_{t_k=0}^\infty B\big(W^{(k)}_{t_k},1\big)\Big)
\]
be the set of points in $\R^d$ not contained in any Wiener sausage of radius $1$ around the Brownian motions $W^{(1)},\ldots, W^{(K)}$. We refer to $\mathcal V_K$ as the \emph{vacant set}.

\smallskip

In this section we study the probability that the vacant set $\mathcal V_K$ contains at least $2$ connected components in each of several well separated microscopic balls. 

\begin{theorem}\label{thm:uniqueness-severalballs}
Let $J\geq 1$ and $x_1,\ldots, x_J\in\R^d$ such that $\|x_j - x_{j'}\|>6$ for all $j\neq j'$. 

Let $\varepsilon\in(0,1)$ and $K\geq 1$. For $1\leq j\leq J$, let $A_{K,j}(W^{(1)},\ldots, W^{(K)})$ be the event that $\mathcal V_K(W^{(1)},\ldots, W^{(K)})\cap B(x_j,\varepsilon)$ contains at least $2$ connected components. 

Then there exist $C=C(d)$ and $m=m(d)$ such that for all $\varepsilon$, $K$, $x_1,\ldots, x_J$ and for all $y_1,\ldots, y_K\in \bigcup\limits_{j=1}^J \partial B(x_j,2)$, 
\begin{equation}\label{eq:uniqueness-severalballs}
\mathsf P_{y_1,\ldots, y_K}\Big[\bigcap\limits_{j=1}^J A_{K,j}(W^{(1)},\ldots, W^{(K)})\Big]\leq \Big(C\log^m\big(\tfrac1\varepsilon\big)\varepsilon^{d+1}\Big)^J.
\end{equation}
\end{theorem}

In Section~\ref{sec:uniqueness-WS-oneball-hemiballs} we show in a deterministic result how to reduce the nonuniqueness question for one ball to a question about hitting all $\varepsilon$-hemiballs of a ball by Brownian motions (see Lemma~\ref{l:uniqueness-hemiball}), which allows to deduce a version of Theorem~\ref{thm:uniqueness-severalballs} for one ball ($J=1$) from Theorem~\ref{thm:hitting-all-hemiballs}.
In order to decouple correlations between events in \eqref{eq:uniqueness-severalballs}, caused by Brownian motions that visit multiple $(1+\varepsilon)$-balls, we prove a local version of the result for one ball, formulated in terms of Brownian excursions (see Theorem~\ref{thm:uniqueness:excursions} in Section~\ref{sec:uniqueness-excursions}). Having done this, Theorem~\ref{thm:uniqueness-severalballs} follows rather directly by an application of the strong Markov property (see Section~\ref{sec:uniqueness-severalballs-proof}). 

\subsection{Case of one ball: Nonuniqueness and hitting of hemiballs}\label{sec:uniqueness-WS-oneball-hemiballs}

In this section we make the first step in the study of nonuniqueness for one ball ($J=1$). 
We prove a deterministic result that the nonuniqueness event implies that all $\varepsilon$-hemiballs of $B(1+\varepsilon)$ (recall Definition~\ref{def:hemiball}) have to be visited by the Brownian motions. This will enable us to apply Theorem~\ref{thm:hitting-all-hemiballs} in the next subsection.

\smallskip

We write $\mathcal V_K$ for $\mathcal V_K(W^{(1)},\ldots, W^{(K)})$.

\begin{lemma}\label{l:uniqueness-hemiball}
Let $\varepsilon\in(0,1)$. Let $W^{(1)},\ldots, W^{(K)}$ be Brownian motions. If $\mathcal V_K\cap B(\varepsilon)$ contains at least $2$ connected components, then 
\begin{enumerate}
\item
every $A_{e,\varepsilon}(1+\varepsilon)$ is hit by at least one of the Brownian motions $W^{(1)},\ldots, W^{(K)}$;
\item
$B(1-\varepsilon)$ is not visited by any of the Brownian motions $W^{(1)},\ldots, W^{(K)}$.
\end{enumerate}
\end{lemma}
\begin{proof}
Statement (2) is trivial, since if a Brownian motion visits $B(1-\varepsilon)$, then the respective Wiener sausage completely covers $B(\varepsilon)$, but $\mathcal V_K\cap B(\varepsilon)$ is non-empty by assumption. 

\smallskip

We proceed with the proof of statement (1). 
Assume on the contrary that, for some $e \in\R^d$ with $\|e\|=1$, $\varepsilon$-hemiball $A_{e,\varepsilon}(1+\varepsilon)$ is not hit by the Brownian motions $W^{(1)},\ldots, W^{(K)}$.

Let $x,y\in \mathcal V_K\cap B(\varepsilon)$. We want to show that $x$ and $y$ belong to the same connected component of $\mathcal V_K\cap B(\varepsilon)$.

First note that if $x\in\mathcal V_K\cap B(\varepsilon)$, then any $\widetilde x = x + \mu e$ with $\mu>0$ contained in $B(\varepsilon)$ is also vacant: $\widetilde x\in\mathcal V_K$. Indeed, it suffices to prove that for any $a\notin A_{e,\varepsilon}(1+\varepsilon)$ such that $B(a,1)\not\ni x$, it also holds that $B(a,1)\not\ni \widetilde x$; in fact, it suffices to prove this for $a$ such that $\langle a,e\rangle \leq -\varepsilon$, since otherwise $B(a,1)\cap B(\varepsilon)=\emptyset$. Assume that $B(a,1)\not\ni x$, that is $\|x-a\|>1$. Then 
\begin{eqnarray*}
\|\widetilde x-a\|^2 &= &\langle \widetilde x-a,\widetilde x-a\rangle = 
\langle x+ \mu e - a, x+\mu e - a\rangle\\
&= &\|x-a\|^2 + \mu^2 + 2\mu\big(\langle x,e\rangle - \langle a,e\rangle\big).
\end{eqnarray*}
Since $x\in B(\varepsilon)$, we have $\langle x,e\rangle\geq -\varepsilon$; thus, $\langle x,e\rangle - \langle a,e\rangle\geq 0$. As a result, $\|\widetilde x-a\|\geq \|x-a\|>1$. 

Thus, it suffices to prove that any two points of $\mathcal V_K$ that lie on the ``upper'' hemisphere of $\partial B(\varepsilon)$ in the direction of $e$ are connected in $\mathcal V_K\cap B(\varepsilon)$:

Let $x,y\in\mathcal V_K\cap\partial B(\varepsilon)$ such that $\langle x,e\rangle\geq 0$ and $\langle y,e\rangle\geq 0$. It suffices to prove that $x$ and $y$ are connected in $\mathcal V_K\cap \partial B(\varepsilon)$. 

By convexity, for any $\lambda\in(0,1)$, $\lambda x+ (1-\lambda)y\in B(\varepsilon)$ and there is unique $\mu_\lambda>0$ such that 
\[
z_\lambda = \lambda x + (1-\lambda) y + \mu_\lambda e\in\partial B(\varepsilon). 
\]
Since the set $\{z_\lambda\,:\,\lambda\in[0,1]\}\subseteq\partial B(\varepsilon)$ is connected and contains $x$ and $y$, it suffices to prove that for every $\lambda\in(0,1)$, $z_\lambda\in\mathcal V_K$; this would follow from the following: 
\begin{equation}\label{eq:zlambda-vacant}
\begin{array}{c}\text{For every $a\in\R^d$ with $\langle a,e\rangle\leq -\varepsilon$, if $\|x-a\|>1$ and $\|y-a\|>1$,}\\ 
 \text{then $\|z_\lambda-a\|>1$ for all $\lambda\in(0,1)$.}
\end{array}
\end{equation}
We have for any $\overline a\in\R^d$ (using $2\langle x-\overline a,y-\overline a\rangle = \|x-\overline a\|^2 + \|y-\overline a\|^2 - \|x-y\|^2$)
\begin{eqnarray*}
\|z_\lambda-\overline a\|^2 &= &\langle z_\lambda - \overline a,z_\lambda - \overline a\rangle\\
&= &\lambda\|x-\overline a\|^2 + (1-\lambda)\|y-\overline a\|^2 - 2\mu_\lambda\langle \overline a,e\rangle\\
&+ &\mu_\lambda^2 - \lambda(1-\lambda) \|x-y\|^2 + 2\lambda\mu_\lambda\langle x,e\rangle + 2(1-\lambda)\mu_\lambda\langle y,e\rangle .
\end{eqnarray*}
When $\overline a=0$, we get (since $x,y,z_\lambda\in\partial B(\varepsilon)$)
\begin{eqnarray*}
\varepsilon^2 &= &\|z_\lambda\|^2 \\
&= &\underbrace{\lambda\|x\|^2 + (1-\lambda)\|y\|^2}_{=\varepsilon^2} + \mu_\lambda^2 - \lambda(1-\lambda) \|x-y\|^2 + 2\lambda\mu_\lambda\langle x,e\rangle + 2(1-\lambda)\mu_\lambda\langle y,e\rangle.
\end{eqnarray*}
Thus, 
\[
\mu_\lambda^2 - \lambda(1-\lambda) \|x-y\|^2 + 2\lambda\mu_\lambda\langle x,e\rangle + 2(1-\lambda)\mu_\lambda\langle y,e\rangle = 0
\]
and the above equality becomes
\[
\|z_\lambda-\overline a\|^2 = \lambda\|x-\overline a\|^2 + (1-\lambda)\|y-\overline a\|^2 - 2\mu_\lambda\langle \overline a,e\rangle, 
\]
from which \eqref{eq:zlambda-vacant} immediately follows. 
The proof of Lemma~\ref{l:uniqueness-hemiball} is completed. 
\end{proof}

\subsection{Case of one ball: Reduction to Theorem~\ref{thm:hitting-all-hemiballs}}\label{sec:nonuniqueness-reduction-to-hemiballs}

In this section we study nonuniqueness in $B(\varepsilon)$ in the case, when the Brownian motions all start from uniform positions on $\partial B(1+\varepsilon)$. We show here how this result follows from Theorem~\ref{thm:hitting-all-hemiballs} about the probability that all $\varepsilon$-hemiballs of $B(1)$ are visited by the Brownian motions. In the subsequent sections we derive Theorem~\ref{thm:uniqueness-severalballs} from this special case for one ball.

\begin{theorem}\label{thm:uniqueness:oneball}
Let $\varepsilon\in(0,1)$. Assume that $W^{(1)},\ldots, W^{(K)}$ start from uniform points on $\partial B(1+\varepsilon)$. 
Then there exist $C=C(d)$ and $m=m(d)$, such that for all $\varepsilon$ and $K$,
\[
\mathsf P\Big[\begin{array}{c}
\text{$\mathcal V_K(W^{(1)},\ldots, W^{(K)})\cap B(\varepsilon)$ contains}\\
\text{at least $2$ connected components}\end{array}
\Big]\leq C\log^m(\tfrac1\varepsilon)\, \varepsilon^{d+1}.
\]
\end{theorem}
\begin{proof}
By Lemma~\ref{l:uniqueness-hemiball}, to prove Theorem~\ref{thm:uniqueness:oneball} it suffices to show that 
if independent Brownian motions $W^{(1)},\ldots, W^{(K)}$ start from uniform points on $\partial B(1+\varepsilon)$, then 
\begin{equation}\label{eq:hit-hemiball-main}
\mathsf P\left[
\begin{array}{l}
\text{every $\varepsilon$-hemiball in $\mathcal A_{\varepsilon}(1+\varepsilon)$ is hit by at least one}\\
\text{of the Brownian motions and $B(1-\varepsilon)$ is not visited}\\ 
\text{by any of the Brownian motions $W^{(1)},\ldots, W^{(K)}$}
\end{array}
\right]\leq C\log^m(\tfrac1\varepsilon)\, \varepsilon^{d+1}.
\end{equation}
This is trivial, when $K\geq d+1$, since by Lemma~\ref{l:BM-hitting},
\begin{eqnarray*}
\mathsf P\left[\begin{array}{c}\text{$B(1-\varepsilon)$ is not visited by any}\\ \text{of the Brownian motions}\end{array}\right] &= &\mathsf P\big[\text{$W^{(1)}$ does not visit $B(1-\varepsilon)$}\big]^K\\
&= &(1 - (1-\varepsilon)^{d-2})^K
\leq (1 - (1-\varepsilon)^{d-2})^{d+1}\\ &\leq &C\varepsilon^{d+1}.
\end{eqnarray*}

When $1\leq K\leq d$, \eqref{eq:hit-hemiball-main} follows from Theorem~\ref{thm:hitting-all-hemiballs} by rescaling. 

Indeed, let $\widetilde W^{(1)} = \tfrac{1}{1+\varepsilon}W^{(1)},\ldots, \widetilde W^{(K)} = \tfrac{1}{1+\varepsilon}W^{(K)}$. By scale invariance of the Brownian motion, $\widetilde W^{(1)},\ldots, \widetilde W^{(K)}$ are independent Brownian motion started from uniform points on $\partial B(1)$. Let $\widetilde \varepsilon = \tfrac{2\varepsilon}{1+\varepsilon}$. Note that 
\[
\tfrac{1}{1+\varepsilon}B(1-\varepsilon) = B(1-\widetilde\varepsilon)\quad\text{and}\quad
\tfrac{1}{1+\varepsilon}A_{e,\varepsilon}(1+\varepsilon) = A_{e,\tfrac12\widetilde\varepsilon}(1)\subseteq A_{e,\widetilde\varepsilon}(1).
\]
Hence, the probability in \eqref{eq:hit-hemiball-main} is bounded from above by 
\[
\mathsf P\left[
\begin{array}{l}
\text{every $\widetilde\varepsilon$-hemiball in $\mathcal A_{\widetilde\varepsilon}(1)$ is hit by at least one of the}\\
\text{Brownian motions $\widetilde W^{(1)},\ldots, \widetilde W^{(K)}$ and $B(1-\widetilde\varepsilon)$ is not}\\ 
\text{visited by any of the Brownian motions $\widetilde W^{(1)},\ldots, \widetilde W^{(K)}$}
\end{array}
\right].
\]
By Theorem~\ref{thm:hitting-all-hemiballs}, the above probability is bounded from above by $C\log^m(\tfrac{1}{\widetilde\varepsilon})\, (\widetilde \varepsilon)^{d+1}$. Since $\varepsilon\leq \widetilde\varepsilon\leq 2\varepsilon$, \eqref{eq:hit-hemiball-main} follows. The proof is completed.
\end{proof}

\subsection{Case of one ball: Local version}\label{sec:uniqueness-excursions}

For $K\geq 1$, let $W^{(1)}, \ldots, W^{(K)}$ be independent Brownian motions in $\R^d$ ($d\geq 3$) starting on $\partial B(2)$. 
For $r>0$, let $T^{(k)}_r = \inf\{t\geq 0\,:\,W^{(k)}_t\in\partial B(r)\}$. 
Let
\[
\widehat{\mathcal V}_K(W^{(1)},\ldots, W^{(K)}) = \R^d\setminus\Big(\bigcup\limits_{k=1}^K\bigcup\limits_{t_k=0}^{T^{(k)}_3} B\big(W^{(k)}_{t_k},1\big)\Big)
\]
be the set of points in $\R^d$ not contained in any Wiener sausage of radius $1$ around the Brownian motions $W^{(1)},\ldots, W^{(K)}$ stopped upon the first exit time from $B(3)$.

\begin{theorem}\label{thm:uniqueness:excursions}
Let $\varepsilon\in(0,1)$. Let $\widehat A_K(W^{(1)},\ldots, W^{(K)})$ be the event that the number of connected components in $\widehat{\mathcal V}_K(W^{(1)},\ldots, W^{(K)})\cap B(\varepsilon)$ is at least $2$. 
Then there exist $C=C(d)$ and $m=m(d)$, such that for all $\varepsilon$, $K$ and $y_1,\ldots, y_K\in\partial B(2)$,
\[
\mathsf P_{y_1,\ldots, y_K}\big[\widehat A_K(W^{(1)},\ldots, W^{(K)})\big]\leq C\log^m(\tfrac1\varepsilon)\, \varepsilon^{d+1}.
\]
\end{theorem}
\begin{proof}
Let $N$ be the number of Brownian motions that hit $B(1+\varepsilon)$ before leaving $B(3)$. 
Note that $N\geq 1$ if $\widehat A_K(W^{(1)},\ldots, W^{(K)})$ occurs. 
We have 
\begin{multline*}
\mathsf P_{y_1,\ldots, y_K}\big[\widehat A_K(W^{(1)},\ldots, W^{(K)})\big] = \sum\limits_{k=1}^K \mathsf P_{y_1,\ldots, y_K}\big[\widehat A_K(W^{(1)},\ldots, W^{(K)}), N=k\big]\\ 
\begin{aligned}
= &\sum\limits_{k=1}^K\sum\limits_{\substack{I\subseteq\{1,\ldots, K\}\\|I|=k}} \mathsf P_{y_1,\ldots, y_K}\big[\widehat A_K(W^{(1)},\ldots, W^{(K)}),T^{(i)}_{1+\varepsilon}<T^{(i)}_3\text{ iff }i\in I\big]\\
= &\sum\limits_{k=1}^K\sum\limits_{\substack{I\subseteq\{1,\ldots, K\}\\|I|=k}} \mathsf P_{y_1,\ldots, y_K}\big[T^{(i)}_{1+\varepsilon}<T^{(i)}_3\text{ iff }i\in I\big]\\
&\qquad\qquad\qquad \mathsf P_{y_1,\ldots, y_K}\big[\widehat A_K(W^{(1)},\ldots, W^{(K)})\,|\,T^{(i)}_{1+\varepsilon}<T^{(i)}_3\text{ iff }i\in I\big].
\end{aligned}
\end{multline*}
Let $I=\{i_1,\ldots, i_k\}$ and define $z_1=y_{i_1},\ldots, z_k=y_{i_k}$. By the locality of event $\widehat A$ and the strong Markov property at the first hitting time of $B(1+\varepsilon)$, 
\begin{multline*}
\mathsf P_{y_1,\ldots, y_K}\big[\widehat A_K(W^{(1)},\ldots, W^{(K)})\,|\,T^{(i)}_{1+\varepsilon}<T^{(i)}_3\text{ iff }i\in I\big]\\
\begin{aligned}
= 
&\mathsf P_{z_1,\ldots,z_k}\big[\widehat A_k(W^{(1)},\ldots, W^{(k)})\,|\,T^{(i)}_{1+\varepsilon}<T^{(i)}_3\text{ for all }1\leq i\leq k\big]\\
= &\int\limits_{(\partial B(1+\varepsilon))^k}\bigotimes\limits_{i=1}^k Q_i(dx_i)\,\mathsf P_{x_1,\ldots, x_k}\big[\widehat A_k(W^{(1)},\ldots, W^{(k)})\big],
\end{aligned}
\end{multline*}
where 
\[
Q_i[\cdot] = \mathsf P_{z_i}\big[W_{T_{1+\varepsilon}}\in \cdot\,|\,T_{1+\varepsilon}<T_3\big].
\]

\smallskip

If $k\geq d+1$, then (as in the proof of Theorem~\ref{thm:uniqueness:oneball}) there exists $C=C(d)$ such that 
\[
\mathsf P_{x_1,\ldots, x_k}\big[\widehat A_k(W^{(1)},\ldots, W^{(k)})\big]
\leq \mathsf P_{x_1,\ldots, x_k}\big[T^{(i)}_3<T^{(i)}_{1-\varepsilon}\text{ for }1\leq i\leq k\big]
\leq C\varepsilon^{d+1}.
\]
Thus, for all $k\geq d+1$, 
\[
\mathsf P_{y_1,\ldots, y_K}\big[\widehat A_K(W^{(1)},\ldots, W^{(K)})\,|\,T^{(i)}_{1+\varepsilon}<T^{(i)}_3\text{ iff }i\in I\big]\leq C\varepsilon^{d+1}.
\]

Now, assume that $k\leq d$. 
Let $\overline Q$ be the uniform distribution on $\partial B(1+\varepsilon)$. By the Poisson formula (Lemma~\ref{l:poisson-formula}) and Lemma~\ref{l:BM-hitting}, there exists $C=C(d)$ such that for any $B\in\mathscr B(\partial B(1+\varepsilon))$, 
\[
Q_i[B] \leq \frac{\mathsf P_{z_i}\big[W_{T_{1+\varepsilon}}\in B \big]}{\mathsf P_{z_i}\big[T_{1+\varepsilon}<T_3\big]}
\leq C\frac{|B|}{|\partial B(1+\varepsilon)|} = C\,\overline Q(B).
\]
Furthermore, let 
\[
\mathcal V_k(W^{(1)},\ldots, W^{(k)}) = \R^d\setminus\Big(\bigcup\limits_{i=1}^k\bigcup\limits_{t_i=0}^\infty B\big(W^{(i)}_{t_i},1\big)\Big)
\]
and $A_k(W^{(1)},\ldots, W^{(k)})$ the event that set $\mathcal V_k(W^{(1)},\ldots, W^{(k)})\cap B(\varepsilon)$ consists of at least $2$ connected components. If $\widetilde T^{(i)}_{1+\varepsilon} = \inf\{t>T^{(i)}_3\,:\,W^{(i)}_t\in B(1+\varepsilon)\}$, then for some $C=C(d)$, 
\begin{multline*}
\mathsf P_{x_1,\ldots, x_k}\big[\widehat A_k(W^{(1)},\ldots, W^{(k)})\big]\\
\begin{aligned}
= 
&\quad\frac{\mathsf P_{x_1,\ldots, x_k}\big[\widehat A_k(W^{(1)},\ldots, W^{(k)}), \widetilde T^{(i)}_{1+\varepsilon}=+\infty\text{ for all }1\leq i\leq k\big]}{\mathsf P_{x_1,\ldots, x_k}\big[\widetilde T^{(i)}_{1+\varepsilon}=+\infty\text{ for all }1\leq i\leq k\big]}\\
= &\quad\frac{\mathsf P_{x_1,\ldots, x_k}\big[A_k(W^{(1)},\ldots, W^{(k)}), \widetilde T^{(i)}_{1+\varepsilon}=+\infty\text{ for all }1\leq i\leq k\big]}{\mathsf P_{x_1,\ldots, x_k}\big[\widetilde T^{(i)}_{1+\varepsilon}=+\infty\text{ for all }1\leq i\leq k\big]}\\
\leq &\quad C^k\,\mathsf P_{x_1,\ldots, x_k}\big[A_k(W^{(1)},\ldots, W^{(k)})\big],
\end{aligned}
\end{multline*}
where the last inequality follows from Lemma~\ref{l:BM-hitting}.

Hence, for some $C=C(d)$,
\begin{multline*}
\mathsf P_{y_1,\ldots, y_K}\big[\widehat A_K(W^{(1)},\ldots, W^{(K)})\,|\,T^{(i)}_{1+\varepsilon}<T^{(i)}_3\text{ iff }i\in I\big]\\
\leq C^{2k}\int\limits_{(\partial B(1+\varepsilon))^k}\bigotimes\limits_{i=1}^k \overline Q(dx_i)
\,\mathsf P_{x_1,\ldots, x_k}\big[A_k(W^{(1)},\ldots, W^{(k)})\big].
\end{multline*}
By Theorem~\ref{thm:uniqueness:oneball}, the last integral is bounded from above by $C\log^m\big(\tfrac1\varepsilon\big)\varepsilon^{d+1}$. Thus, for $k\leq d$, we obtain for some $C=C(d)$ and $m=m(d)$ that 
\[
\mathsf P_{y_1,\ldots, y_K}\big[\widehat A_K(W^{(1)},\ldots, W^{(K)})\,|\,T^{(i)}_{1+\varepsilon}<T^{(i)}_3\text{ iff }i\in I\big]\leq C\log^m\big(\tfrac1\varepsilon\big)\varepsilon^{d+1}.
\]

All in all, we obtain that 
\begin{multline*}
\mathsf P_{y_1,\ldots, y_K}\big[\widehat A_K(W^{(1)},\ldots, W^{(K)})\big]\\
\begin{aligned}
\leq &\Big( C\log^m\big(\tfrac1\varepsilon\big)\varepsilon^{d+1}\Big)\,
\sum\limits_{k=1}^K\sum\limits_{\substack{I\subseteq\{1,\ldots, K\}\\|I|=k}} \mathsf P_{y_1,\ldots, y_K}\big[T^{(i)}_{1+\varepsilon}<T^{(i)}_3\text{ iff }i\in I\big]\\
= 
&\Big( C\log^m\big(\tfrac1\varepsilon\big)\varepsilon^{d+1}\Big)\,\mathsf P_{y_1,\ldots,y_K}[N\geq 1]
\leq C\log^m\big(\tfrac1\varepsilon\big)\varepsilon^{d+1}.
\end{aligned}
\end{multline*}
The proof of Theorem~\ref{thm:uniqueness:excursions} is completed. 
\end{proof}

\subsection{Proof of Theorem~\ref{thm:uniqueness-severalballs}}\label{sec:uniqueness-severalballs-proof}

Fix $J\geq 1$ and $x_1,\ldots, x_J\in\R^d$ such that $\|x_j - x_{j'}\|>6$ for all $j\neq j'$, and
define 
\[
B = \bigcup\limits_{j=1}^J B(x_j,2)\quad\text{and}\quad
B' = \bigcup\limits_{j=1}^J B(x_j,3).
\]
By assumption, each Brownian motions starts on $\partial B$. 
We decompose the path of each Brownian motion into excursions from $\partial B$ to $\partial B'$. 
The event $A_{K,j}(W^{(1)},\ldots, W^{(K)})$ only depends on the excursions from $\partial B(x_j,2)$ to $\partial B(x_j,3)$. So, we can apply Theorem~\ref{thm:uniqueness:excursions} to the set of excursions for each $j$.

For $1\leq k\leq K$, consider the stopping times $\tau^{(k)}_1 = 0$, $\eta^{(k)}_1 = \inf\{t>\tau^{(k)}_1\,:\,W^{(k)}_t\notin B'\}$, and for $i\geq 2$, 
\[
\tau^{(k)}_i = \inf\{t>\eta^{(k)}_{i-1}\,:\,W^{(k)}_t\in B\},\quad 
\eta^{(k)}_i = \inf\{t>\tau^{(k)}_i\,:\,W^{(k)}_t\notin B'\}
\]
and define $N^{(k)}_j = \big|\{i\,:\,\tau^{(k)}_i<\infty\text{ and }W^{(k)}_{\tau^{(k)}_i}\in B_j\}\big|$.
Then
\begin{multline*}
\mathsf P_{y_1,\ldots, y_K}\Big[\bigcap\limits_{j=1}^J A_{K,j}(W^{(1)},\ldots, W^{(K)})\Big]\\
= \sum\limits_{n_1,\ldots, n_J\geq 1}
\mathsf P_{y_1,\ldots, y_K}\Big[\bigcap\limits_{j=1}^J A_{K,j}(W^{(1)},\ldots, W^{(K)}),\,\sum\limits_{k=1}^KN^{(k)}_j = n_j\text{ for }1\leq j\leq J\Big]\\
\begin{aligned}
\leq 
&\sum\limits_{n_1,\ldots, n_J\geq 1}
\mathsf P_{y_1,\ldots, y_K}\Big[\sum\limits_{k=1}^KN^{(k)}_j = n_j\text{ for }1\leq j\leq J\Big]\\
&\mathsf P_{y_1,\ldots, y_K}
\Big[\bigcap\limits_{j=1}^J A_{K,j}(W^{(1)},\ldots, W^{(K)})\,\Big|\,\sum\limits_{k=1}^KN^{(k)}_j = n_j\text{ for }1\leq j\leq J\Big].
\end{aligned}
\end{multline*}
Let $\widehat A_{k,j}(W^{(1)},\ldots, W^{(k)})$ be the analogue of event $\widehat A_k(W^{(1)},\ldots, W^{(k)})$ from Theorem~\ref{thm:uniqueness:excursions} for the balls centered in $x_j$. By the strong Markov property, we obtain that 
\begin{multline*}
\mathsf P_{y_1,\ldots, y_K}
\Big[\bigcap\limits_{j=1}^J A_{K,j}(W^{(1)},\ldots, W^{(K)})\,\Big|\,\sum\limits_{k=1}^KN^{(k)}_j = n_j\text{ for }1\leq j\leq J\Big]\\
\leq
\prod\limits_{j=1}^J \sup\limits_{\substack{x_{j,1},\ldots, x_{j,n_j}\\\in\partial B(x_j,2)}} \mathsf P_{x_{j,1},\ldots, x_{j,n_j}}\big[\widehat A_{n_j,j}(W^{(1)},\ldots, W^{(n_j)})\big].
\end{multline*}
By Theorem~\ref{thm:uniqueness:excursions}, each probability on the right hand side is bounded from above by $C\log^m\big(\tfrac1\varepsilon\big)\varepsilon^{d+1}$ for some $C=C(d)$ and $m=m(d)$. The result follows. \qed

\section{Brownian interlacements}\label{sec:BI}

In this section we define the model of Brownian interlacements on $\R^d$, $d\geq 3$, precisely. 
For further details and proofs we refer the reader to \cite[Section~2]{Sznitman-BI}.

\subsection{Basics of Brownian motion and potential theory}

Let $\mathsf W_+$ be the space of continuous $\R^d$-valued paths tending to infinity at infinite times. We denote by $X_t$, $t\geq 0$, the canonical process on $\mathsf W_+$  (i.e.\ $X_t(w) = w(t)$) and by $\mathcal W_+$ the sigma-algebra on $\mathsf W_+$ generated by the canonical process. 
Since the Brownian motion is transient in dimension $d\geq 3$, the law of the Brownian motion starting from $x\in\R^d$, denoted by $\mathsf P_x$, is a probability measure on $(\mathsf W_+,\mathcal W_+)$.

\smallskip

We write 
\[
p_t(x,x') = \tfrac{1}{(2\pi t)^{\frac d2}}\,\exp\big(-\tfrac{\|x-x'\|^2}{2t}\big),\quad x,x'\in\R^d,\,t>0,
\]
for the Brownian transition density and 
\[
G(x,x') = \int\limits_0^\infty p_t(x,x')\,dt,\quad x,x'\in\R^d,
\]
for the respective Green function. 

\smallskip

Given a compact set $K$ in $\R^d$, we denote by $e_K$ the equilibrium measure of $K$, which is a finite measure uniquely determined by the last visit formula (see \cite[Theorem 3.4]{Sznitman-Book}):
\begin{equation}\label{eq:last-visit-formula}
\mathsf P_x\big[L_K>0,\,L_K\in dt,\,X_{L_K}\in dy\big] = p_t(x,y)\,e_K(dy)\,dt,
\end{equation}
where 
\[
L_K(w) = \sup\{t>0\,:\,X_t(w)\in K\},\quad w\in \mathsf W_+,
\]
is the time of the last visit of $w$ to $K$. 
The total mass of $e_K$ is called the capacity of $K$ and is denoted by $\mathrm{cap}(K)$. 

\smallskip

For a closed ball $B$ and $x\notin \mathrm{int}(B)$, we denote by $\mathsf P^B_x$ the law of the Brownian motion starting from $x$ and conditioned on staying outside of $B$ for all positive times.
(For $x\in\partial B$, one can make sense of $\mathsf P^B_x$, for example, via a weak limit procedure; see \cite[Theorems 4.1 and 2.2.]{Burdzy-87}.)

\smallskip

We denote by $\mathsf P_{x,y}^t$ the Brownian bridge measure in time $t>0$ from $x$ to $y$, see \cite[p. 137--140]{Sznitman-Book}. 
By \cite[Theorem~2.12]{Getoor}, for any closed ball $B$ and $x\in B$, 
\begin{equation}\label{eq:conditional-independence-exit-time}
\begin{array}{c}
\text{under $\mathsf P_x$, conditionally on $L_B=t>0$ and $X_{L_B}=y\in \partial B$, the processes}\\ 
\text{$(X_s)_{0\leq s\leq t}$ and $(X_{L_B+s})_{s\geq 0}$ are independent and have laws $\mathsf P^t_{x,y}$ resp.\ $\mathsf P^B_y$.}
\end{array}
\end{equation}
(Note that the joint law of $L_B$ and $X_{L_B}$ is explicit and given by \eqref{eq:last-visit-formula}.)

\subsection{Compatible measures on doubly-infinite paths}
Let $\mathsf W$ be the space of continuous doubly-infinite $\R^d$-valued paths tending to infinity at positive and negative infinite times. We denote by $X_t$, $t\in\R$, the canonical process on $\mathsf W$  (i.e.\ $X_t(w) = w(t)$) and by $\mathcal W$ the sigma-algebra on $\mathsf W$ generated by the canonical process. We denote the canonical time shift on $\mathsf W$ by $\theta_t$, $t\in\R$. For a closed set $F\subseteq \R^d$, we define the first entrance time of $w\in\mathsf W$ in $F$ as 
$H_F(w)= \inf\{t\in\R\,:\,X_t(w)\in F\}$.

\smallskip

For a compact set $K$ in $\R^d$, we write 
\[
\mathsf W_K = \{w\in \mathsf W\,;\,H_K(w)<\infty\},\quad 
\mathsf W_K^0 = \{w\in \mathsf W\,:\, H_K(w) = 0\}
\]
for the sets of paths that ever visit $K$, resp., visit $K$ for the first time at time $0$. 

\smallskip

For a closed ball $B$ of positive radius, we define the following measure on $\mathsf W_B^0$:
\begin{equation}\label{def:Q_B}
Q_B\big[(X_{-t})_{t\geq 0}\in A,\,X_0\in dx,\,(X_t)_{t\geq 0}\in A'\big] = e_B(dx)\,\mathsf P^B_x[A]\,\mathsf P_x[A'],\quad A,A'\in \mathcal W_+.
\end{equation}
The measures $Q_B$ are compatible, in the sense that $Q_B = \theta_{H_B}\circ\big(\mathds{1}_{\mathsf W_B}\,Q_{B'}\big)$, for any closed balls $B$ and $B'$ with $B\subset\mathrm{int}(B')$, see \cite[Lemma 2.1]{Sznitman-BI}. In particular, for any compact set $K$ in $\R^d$, the measure 
\[
Q_K = \theta_{H_K}\circ\big(\mathds{1}_{\mathsf W_K}\,Q_B\big),
\]
on $\mathsf W_K^0$, where $B$ is any closed ball containing $K$, is well-defined. Note that $Q_K$ is a finite measure with $Q_K[\mathsf W_K^0] = \mathrm{cap}(K)$. 

\medskip

For a compact set $K$ in $\R^d$ and $w\in \mathsf W_K$, we define the last exit time of $w$ from $K$ by 
\[
L_K(w) = \sup\{t\in \R\,:\,X_t(w)\in K\}. 
\]
By \eqref{eq:last-visit-formula}, \eqref{eq:conditional-independence-exit-time} and \eqref{def:Q_B}, for any closed ball $B$ and $A,A'\in\mathcal W_+$, 
\begin{multline}\label{eq:BB-representation-Q_B}
Q_B\big[(X_{-s})_{s\geq 0}\in A,\, (X_s)_{0\leq s\leq L_B}\in\cdot,\, (X_{L_B+s})_{s\geq 0}\in A'\big]\\
= 
\iint\limits_{\partial B\times \partial B} e_B(dx)\,e_B(dx')\,G(x,x')\,
\mathsf P^B_x[A]\,\Big(\int\limits_0^\infty \frac{p_t(x,x')}{G(x,x')}\,\mathsf P^t_{x,x'}[\cdot]\,dt\Big)\,\mathsf P^B_{x'}[A'],
\end{multline}
where $(X_s)_{0\leq s\leq L_B}$ is viewed as a random element on the space $\mathsf W_{\text{fin}}$ of continuous $\R^d$-valued paths of positive finite duration, equipped with the sigma-algebra induced by the map $(w,t)\in C([0,1],\R^d)\times(0,\infty)\mapsto w(\frac{\cdot}{t})\in \mathsf W_{\text{fin}}$. The identity \eqref{eq:BB-representation-Q_B} states that under $Q_B$, the pieces of the random path before the first entrance time in $B$, after the last visit time in $B$, and between those times are conditionally independent.

\subsection{Brownian interlacement measure}
We now define a suitable sigma-finite measure on doubly-infinite paths, whose restriction to every $\mathsf W_K$ is $Q_K$. 

\smallskip

Two paths $w$ and $w'$ in $\mathsf W$ are called equivalent, if $w'=\theta_t(w)$ for some $t\in\R$. The quotient set of $\mathsf W$ modulo this equivalence relation is denoted by $\mathsf W^*$. The canonical projection $\pi^*:\mathsf W\to \mathsf W^*$ induces the sigma-algebra $\mathcal W^* = \{A\subseteq \mathsf W^*\,:\,(\pi^*)^{-1}(A)\in \mathcal W\}$ on $\mathsf W^*$. For a compact set $K$ in $\R^d$, we denote by $\mathsf W_K^*$ the image of $\mathsf W_K$ under $\pi^*$. 
Note that $\pi^*$ maps bijectively $\mathsf W_K^0$ onto $\mathsf W_K^*$.

\smallskip

By \cite[Theorem 2.2]{Sznitman-BI}, there exists a unique sigma-finite measure $\nu$ on $(\mathsf W^*,\mathcal W^*)$, whose restriction to any $\mathsf W_K^*$ coincides with $Q_K$, more precisely, 
\[
\mathds{1}_{\mathsf W_K^*}\,\nu = \pi^*\circ Q_K,\quad \text{for any compact $K$ in $\R^d$}.
\]
Note that $\nu[\mathsf W^*_K] = Q_K[\mathsf W_K] = \mathrm{cap}(K)$. 

\subsection{Brownian interlacement point process}\label{sec:BIPP}

Consider the space of point measures 
\[
\Omega = \Big\{\omega = \sum\limits_{i\geq 1} \delta_{(w_i^*,\alpha_i)}\,:\,\omega\big(\mathsf W_K^*\times[0,\alpha]\big)<\infty\text{ for all compact set $K\subseteq\R^d$ and $\alpha>0$}\Big\}
\]
on $\mathsf W^*\times\R_+$, endowed with the sigma-algebra $\mathcal A$ generated by the evaluation maps 
\[
\omega\mapsto\omega(E), \quad E\in\mathcal W^*\otimes\mathscr B(\R_+),
\]
and denote by $\P$ the Poisson point measure on $\mathsf W^*\times\R_+$ with intensity $\nu\otimes d\alpha$; the random point measure with law $\P$ is called the \emph{Brownian interlacement point process on $\R^d$}. 

The random variable 
\[
N_{K,\alpha} = N_{K,\alpha}(\omega) = \omega\big(\mathsf W_K^*\times[0,\alpha]\big),
\]
which counts the number of trajectories with labels $\leq \alpha$ (in $\omega$) that visit $K$, has Poisson distribution with parameter $\alpha\mathrm{cap}(K)$. 
For any closed ball $B$, given $N_{B,\alpha} = n$, the $n$ trajectories of the Brownian interlacement point process with labels $\leq \alpha$ that visit $B$ are independent random elements of $\mathsf W^*_B$ with the common distribution $\frac{1}{\mathrm{cap}(B)}\big(\pi^*\circ Q_B\big)$, whose labels are independent uniformly distributed on $[0,\alpha]$. By \eqref{eq:BB-representation-Q_B} each of them can be sampled (independenty) as follows: 
\begin{itemize}
\item
Sample the locations of the first entrance and last visit in $B$, $(X_i,X_i')$, from the distribution 
\begin{equation}\label{eq:sample-1}
\frac{1}{\mathrm{cap}(B)}\,G(x,x')\,e_B(dx)\,e_B(dx');
\end{equation}
\item
Given $X_i=x_i$ and $X_i'=x_i'$, sample independently random paths $\gamma_i$ and $\gamma_i'$ in $\mathsf W_+$ and $\widetilde \gamma_i$ in $W_{\mathrm{fin}}$ respectively from the distributions
\begin{equation}\label{eq:sample-2}
\mathsf P^B_{x_i},\quad \mathsf P^B_{x_i'},\quad \int\limits_0^\infty \frac{p_t(x,x')}{G(x,x')}\,\mathsf P^t_{x,x'}[\cdot]\,dt;
\end{equation}
\item
Let $w_i$ be the concatenation of the time reversal of $\gamma_i$, $\widetilde \gamma_i$ and $\gamma_i'$, so that $w_i(t) = \gamma_i(-t)$ for $t\leq 0$. (Note that $w_i$ is a random path in $\mathsf W_B^0$ with the law $\frac{1}{\mathrm{cap}(B)}Q_B$.)
\item
To get the desired random element of $\mathsf W^*_K\times[0,\alpha]$, we project $w_i$ onto $\mathsf W^*$ and assign it an independent label from the uniform distribution on $[0,\alpha]$. 
\end{itemize}

\smallskip

For any $\alpha>0$, we denote by $\P^\alpha$ the 
push forward measure of $\P$ by the map 
\[
\omega = \sum\limits_{i\geq 1}\delta_{(w_i^*,\alpha_i)}\,\,\mapsto\,\, \iota^\alpha(\omega) = \sum\limits_{i\geq 1:\,\alpha_i\leq \alpha}\delta_{w_i^*}. 
\]
$\P^\alpha$ is a Poisson point measure on $\mathsf W^*$ with intensity $\alpha\nu$. We call $\iota^\alpha$ the Brownian interlacement point process on $\R^d$ at level $\alpha$.

\subsection{Brownian interlacement}
For any $\alpha>0$ and $r\geq 0$, the \emph{Brownian interlacement at level $\alpha$ with radius $r$} is defined as
\begin{equation}\label{eq:BI-definition}
\mathcal I^\alpha_r(\omega) = \bigcup\limits_{i\geq 1:\,\alpha_i\leq \alpha}\bigcup\limits_{t\in\R} B\big(w_i(t),r\big),\quad \omega = \sum\limits_{i\geq 1}\delta_{(w_i^*,\alpha_i)}\in\Omega,
\end{equation}
where $w_i$ is any path in $\pi^{-1}(w_i^*)$. 
By the definition of $\Omega$, any $\mathcal I^\alpha_r$ is a closed set in $\R^d$. 
We write $\mathcal I^\alpha$ for $\mathcal I^\alpha_1$ and call it simply the Brownian interlacement at level $\alpha$. 
The complement of $\mathcal I^\alpha_r$ is called the \emph{vacant set (of Brownian interlacement) at level $\alpha$ and radius $r$} and is denoted by $\mathcal V^\alpha_r$. Again, when $r=1$, we just call it the vacant set at level $\alpha$ and omit $r$ from the notation. Any $\mathcal V^\alpha_r$ is an open set in $\R^d$. 

\smallskip

Any Brownian interlacement $\mathcal I^\alpha_r$ is a measurable map from $(\Omega,\mathcal A)$ to $(\Sigma, \mathcal F)$, where $\Sigma$ is the set of all closed subsets of $\R^d$ and $\mathcal F$ is the sigma-algebra on $\Sigma$ generated by the $\pi$-system $\big\{\{F\in\Sigma\,:\,F\cap K = \emptyset\},\,K\subset \R^d\text{ compact}\big\}$. The law of $\mathcal I^\alpha_r$ on $(\Omega,\mathcal A, \P)$ is the probability measure $Q^\alpha_r$ on $(\Sigma,\mathcal F)$ uniquely determined by the identities
\[
Q^\alpha_r\big[\{F\in\Sigma\,:\,F\cap K = \emptyset\}\big] = \P[\mathcal I^\alpha_r\cap K = \emptyset] = e^{-\alpha\,\mathrm{cap}(B(K,r))},\quad K\subset\R^d\text{ compact}.
\]

\smallskip

Let $t_x$ be the translation in $\R^d$ by $x$. By \cite[Proposition 1.5]{Li-BI}, for any $\alpha>0$ and $r\geq 0$, 
\begin{equation}\label{eq:BI-ergodicity}
\text{$(t_x)_{x\in\R^d}$ is a measure preserving ergodic flow on $(\Sigma,\mathcal F, Q^\alpha_r)$.}
\end{equation}

\smallskip

While $\mathcal I^\alpha_r$ is a $\P$-almost surely connected set for any given $\alpha$ and $r$, the vacant set $\mathcal V^\alpha_r$ undergoes a non-trivial percolation phase transition in $\alpha$ (from $\P$-a.s.\ containing an infinite connected component for $\alpha<\alpha_c$ to $\P$-a.s.\ no infinite component for $\alpha>\alpha_c$), see Corollary~0.2 and Theorem~0.3 in \cite{Li-BI}. Furthermore, by \eqref{eq:BI-ergodicity}, the number of infinite connected components in the vacant set $\mathcal V^\alpha_r$ is constant $\P$-almost surely.

\section{Connectivity of Brownian interlacements}\label{sec:connectivity-BI}

It is well known that the traces of two independent Brownian paths in $\R^d$ intersect almost surely if and only if $d\leq 3$; see e.g.\ \cite[Theorem~9.1]{MP-BM-book}. 
This implies that any pair of trajectories in the support of the Poisson point process $\iota^\alpha$ (i.e.\ interlacement trajectories with labels $\leq \alpha$) a.s.\ intersect when $d= 3$ and a.s.\ do not intersect when $d\geq 4$. In particular, $\mathcal I^\alpha_0$ is a.s.\ connected if and only if $d= 3$. At the same time, for any radius $r>0$, $\mathcal I^\alpha_r$ is a.s.\ connected in any dimension, see \cite[Corollary~1.2]{Li-BI}. In fact, \cite[Theorem~1.1]{Li-BI} proves that a.s.\ any pair of $r$-sausages $B(w_i^*,r)$ and $B(w_j^*,r)$ around interlacement trajectories $w_i^*$ and $w_j^*$ from $\iota^\alpha$ are connected with each other via a chain of at most $\lceil \frac{d-4}2\rceil$ other $r$-sausages around some interlacement trajectories from $\iota^\alpha$. 

\smallskip

In this section, we refine the connectivity property of $\mathcal I^\alpha_r$ by showing that 
any pair of interlacement trajectories from the support of $\iota^\alpha$ are a.s.\ connected inside the ``$r$-interior'' of $\mathcal I^\alpha_r$; in words, for any $x,y\in\mathcal I^\alpha_0$, the ball $B(x,r)$ can be continuously transported to the ball $B(y,r)$ while remaining inside the set $\mathcal I^\alpha_r$. We will use this property in the proof of the uniqueness of the infinite connected component in the vacant set $\mathcal V^\alpha$ to construct certain local modifications by rerouting interlacement trajectories to avoid a given region (see Proposition~\ref{prop:N=k} and Lemma~\ref{l:trifurcation}).

\smallskip

For $r>0$, we call two sets $A_1,A_2\subseteq \R^d$ \emph{$r$-connected} if there exists a continuous path $\gamma:[0,1]\to\R^d$, such that (a) $\gamma(0)\in A_1$, (b) $\gamma(1)\in A_2$ and (c) $B(\gamma(t),r)\subseteq B(A_1\cup A_2,r)$ for all $t$;\footnote{$r$-connectedness of sets is weaker than intersection of sets, but stronger than intersection of $r$-neighborhoods of the sets.}
we write $A_1\stackrel{r}\longleftrightarrow A_2$ if $A_1$ and $A_2$ are $r$-connected. 
In what follows we are interested in $r$-connectedness of trajectories from the support of the interlacement point process $\iota^\alpha$. In Proposition~\ref{prop:BI-connectivity-int}, we abuse notation and denote the range of a path (resp.\ equivalence class of paths) $\gamma$ in $\R^d$ also by $\gamma$. 

\begin{proposition}\label{prop:BI-connectivity-int}
Let $d\geq 3$, $\alpha,r>0$ and $x,y\in\R^d$. Let $\iota^\alpha$ be the Brownian interlacement point process at level $\alpha$. Let $W_x$ and $W_y$ be Brownian motions in $\R^d$ from $x$ resp.\ $y$. We assume that $W_x$, $W_y$ and $\iota^\alpha$ are independent. 
\begin{enumerate}
\item
When $d\in\{3,4\}$, then almost surely $W_x\stackrel{r}\longleftrightarrow W_y$. 
\item
When $d\geq 5$, then almost surely 
there exist $1\leq k\leq \lceil\frac{d-4}2\rceil$ trajectories $w^*_1,\ldots, w^*_k$ in the support of $\iota^\alpha$, such that 
\[
W_x\stackrel{r}\longleftrightarrow w^*_1,\,\,
w^*_1\stackrel{r}\longleftrightarrow w^*_2,\,\,\ldots,\,\,
w^*_{k-1}\stackrel{r}\longleftrightarrow w^*_k,\,\,
w^*_k\stackrel{r}\longleftrightarrow W_y.
\]
\end{enumerate}
\end{proposition}
The proof of Proposition~\ref{prop:BI-connectivity-int} is essentially the same as the proof of (the upper bound of) \cite[Theorem~1.1]{Li-BI} and we only indicate the main (minor) needed modifications. 

\begin{proof}
We first exclude the trivial case of $d\in\{3,4\}$. When $d=3$, then the ranges $W_x$ and $W_y$ intersect almost surely, so the result trivially holds. When $d=4$, then (see e.g.\ \cite[Theorem~6.2]{AZ-1996}) the Wiener sausage $B(W_x,r) = \bigcup\limits_{t=0}^\infty B(W_x(t),r)$ is a.s.\ visible for the independent Brownian motion $W_y$, that is 
$W_y$ intersects $B(W_x,r)$ almost surely (and infinitely often). Conditioned on the sausage $B(W_x,r)$, let $T$ be the first hitting time of it by $W_y$. With universal positive probability, during the time interval $[T,T+1]$ the $r$-sausage of $W_y$ covers the ball of radius $2r$ centered at the first hitting point, that is 
\[
B\big(W_y(T),2r\big)\subseteq \bigcup\limits_{t=T}^{T+1}B\big(W_y(t),r\big).
\]
If this occurs, then the paths $W_x$ and $W_y$ are $r$-connected. 
Since $W_y$ hits the sausage $B(W_x,r)$ a.s.\ infinitely often, the paths $W_x$ and $W_y$ are $r$-connected almost surely. 

\smallskip

Let $d\geq 5$. 
We first briefly recall the strategy for the proof of \cite[Theorem~1.1]{Li-BI}, whose main step (\cite[Proposition~3.19]{Li-BI}) is to show that (in the notation of Proposition~\ref{prop:BI-connectivity-int}) 
almost surely there exist $k\leq \lceil\frac{d-4}2\rceil$ trajectories $w^*_1,\ldots, w^*_k$ in the support of $\iota^\alpha$, such that 
\begin{equation}\label{eq:connectivity-sausages}
\begin{array}{l}
B(W_x,r)\cap B(w^*_1,r)\neq \emptyset,\quad 
B(W_y,r)\cap B(w^*_k,r)\neq \emptyset,\\[4pt]
B(w^*_i,r)\cap B(w^*_{i+1},r)\neq \emptyset\,\,\text{ for all }1\leq i\leq k-1.
\end{array}
\end{equation}
This is proven in essentally the same way as Lemma~4.11 of \cite{RS-12}, where the analogue of \cite[Theorem~1.1]{Li-BI} is proven for the random interlacements on $\Z^d$, and can be roughly summarized as follows: 
\begin{itemize}\itemsep4pt
\item
Let $A_0 = B(W_x,r)$. 
\item
We represent $\iota^\alpha$ as the sum of $s_d = \lceil\frac{d-4}2\rceil$ independent copies $\iota_1,\ldots, \iota_{s_d}$ of the Brownian interlacement point process at level $s_d^{-1}\alpha$. 
\item
Let $\iota_1'$ be the point process of trajectories from $\iota_1$ that visit $A_0$ and denote by $A_1$ the union of all the $r$-sausages around the trajectories from the support of $\iota_1'$; let $\iota_2'$ be the point process of trajectories from $\iota_2$ that visit $A_1$ and denote by $A_2$ the union of $A_1$ and all the $r$-sausages around the trajectories from the support of $\iota_2'$; and so on until the set $A = A_{s_d}$ is constructed. 
\item
Using capacity estimates and a variant of Wiener's test, one shows that the set $A$ is a.s.\ visible for the Brownian motion $W_y$. This implies the existence of $k\leq s_d$ trajectories $w^*_1,\ldots, w^*_k$, with $w^*_i$ in the support of $\iota_i$ (in fact, in the support of $\iota_i'$), such that \eqref{eq:connectivity-sausages} holds. 
\end{itemize}

Only minor adjustments are needed to adapt this procedure to a proof of Proposition~\ref{prop:BI-connectivity-int}:
\begin{itemize}
\item
Let $\widetilde \iota_1$ be the point process of those trajectories from $\iota_1$ that (a) visit $A_0$ and (b) in the unit time interval directly after the first visit to $A_0$\footnote{In the rigorous construction, one considers restrictions of sets $A_k$ to large balls, so that the time of their first visit is well defined.}, the $r$-sausage of the trajectory covers the ball of radius $2r$ around the first entrance point. We denote by $\widetilde A_1$ the union of all the $r$-sausages around the trajectories from the support of $\widetilde \iota_1$. Similarly, we define $\widetilde \iota_2,\ldots ,\widetilde \iota_{s_d}$ and the sets $\widetilde A_2,\ldots, \widetilde A_{s_d}$, where $\widetilde A_k$ is the union of $\widetilde A_{k-1}$ and all the $r$-sausages around the trajectories from the support of $\widetilde \iota_k$.
\item
Although the sets $\widetilde A_k$ are constructed from thinner clouds of trajectories than $A_k$, one can control their hittability by the same capacity estimates as in \cite{Li-BI}. The only (minor) difference is in the proof of the analogue of \cite[Lemma~2.10]{Li-BI}, which states that the expected capacity of the restriction to $B(R)$ of the union of the $r$-sausages around interlacement trajectories from $\iota_k$ that intersect a compact set $D (\subset B(\frac12R))$ is $\geq c\min\big(\mathrm{cap}(D)R^2,R^{d-2}\big)$. The same estimate also holds if we thin out the interlacement process $\iota_k$ by keeping only the trajectories, such that in the unit time interval directly after the first hitting time of $D$ the $r$-sausage of the trajectory covers the ball of radius $2r$ centered at the first hitting point, since the number of such trajectories remains to be a Poisson random variable with comparable parameter $c'\mathrm{cap}(D)$.
\item
Having shown that the set $\widetilde A_{s_d} \big(=\bigcup\limits_{i=1}^{s_d}\bigcup\limits_{w^*_i\in\mathrm{supp}(\widetilde\iota_i)}B(w^*_i,r)\big)$ is a.s.\ visible for $W_y$ (cf.\ \cite[Propositions~2.18 and 2.19]{Li-BI}), one concludes similarly as in $d=4$ that 
$W_y$ is almost surely $r$-connected to the union of all the trajectories from the support of $\sum\limits_{i=1}^{s_d}\widetilde\iota_i$. By the construction of $\widetilde\iota_i$'s, this implies the existence of $1\leq k\leq s_d$ trajectories $w^*_1,\ldots, w^*_k$, with $w^*_i$ in the support of $\widetilde \iota_i$, such that $W_x\stackrel{r}\longleftrightarrow w^*_1$, $w^*_1\stackrel{r}\longleftrightarrow w^*_2$, $\ldots$, 
$w^*_{k-1}\stackrel{r}\longleftrightarrow w^*_k$, $w^*_k\stackrel{r}\longleftrightarrow W_y$. 
Thus, statement (2) of the proposition holds. 
\end{itemize}
\end{proof}

Since the Brownian motion is transient in dimensions $d\geq 3$, 
we obtain from Proposition~\ref{prop:BI-connectivity-int} its following analogue for Brownian motions conditioned to avoid a closed ball. 

\begin{corollary}\label{cor:condBI-connectivity-int}
Let $d\geq 3$, $\alpha,r,K>0$ and $x,y\in\partial B(K)$. Let $\iota^\alpha$ be the Brownian interlacement point process at level $\alpha$. Let $W_x^K$ and $W_y^K$ be Brownian motions in $\R^d$ from $x$ resp.\ $y$ conditioned on never revisiting $B(K)$ at positive times. We assume that $W_x^K$, $W_y^K$ and $\iota^\alpha$ are independent. 
\begin{enumerate}
\item
When $d\in\{3,4\}$, then almost surely $W_x^K\stackrel{r}\longleftrightarrow W_y^K$. 
\item
When $d\geq 5$, then almost surely 
there exist $1\leq k\leq \lceil\frac{d-4}2\rceil$ trajectories $w^*_1,\ldots, w^*_k$ in the support of $\iota^\alpha$, such that 
\[
W_x^K\stackrel{r}\longleftrightarrow w^*_1,\,\,
w^*_1\stackrel{r}\longleftrightarrow w^*_2,\,\,\ldots,\,\,
w^*_{k-1}\stackrel{r}\longleftrightarrow w^*_k,\,\,
w^*_k\stackrel{r}\longleftrightarrow W_y^K.
\]
\end{enumerate}
\end{corollary}

\section{Uniqueness of the infinite connected component for the vacant set of Brownian interlacements}\label{sec:uniqueness-proof}

In this section we prove the uniqueness result for the number of infinite connected components in the vacant set of Brownian interlacements. Recall the definition of the vacant set $\mathcal V^\alpha_r$ from below \eqref{eq:BI-definition}.

\begin{theorem}\label{thm:uniqueness-BI}
For any $\alpha>0$ and $r>0$, the number of infinite connected components in $\mathcal V^\alpha_r$ is either a.s.\ equal to $0$ or a.s.\ equal to $1$.
\end{theorem}
By \cite[(2.35)]{Sznitman-BI}, for any positive $\alpha, r,\lambda$,
\[
\lambda \mathcal V^\alpha_r = \mathcal V^{\lambda^{2-d}\alpha}_{\lambda r}.
\]
Thus, it suffices to prove Theorem~\ref{thm:uniqueness-BI} for $r=1$. Recall that we denote $\mathcal V^\alpha_1$ as $\mathcal V^\alpha$. 

\smallskip

We fix $\alpha>0$ and let $N$ be the number of infinite connected components in $\mathcal V^\alpha$. By \eqref{eq:BI-ergodicity}, $N$ is constant almost surely. 
Theorem~\ref{thm:uniqueness-BI} follows from Propositions~\ref{prop:N=k} and \ref{prop:N=01}, in which we show that $\P[N=k]=0$ for all $2\leq k<\infty$, respectively, $\P[N=\infty]=0$. In Section~\ref{sec:uniqueness-notation}, we collect some common notation that we use for local modification arguments (in the proofs of Lemma~\ref{l:N=k-complement}, Proposition~\ref{prop:N=k} and Lemma~\ref{l:trifurcation}); we also prove there in Proposition~\ref{prop:r-connectivity-IKn} a connectedness result for Brownian interlacements, which is crucially used in the local modification arguments of Proposition~\ref{prop:N=k} and Lemma~\ref{l:trifurcation}. Section~\ref{sec:expected-number} contains a result about expected number of large connected components of $\mathcal V^\alpha$ in a ball, which is needed to perform the Burton-Keane argument (see the proof of Proposition~\ref{prop:N=01}).

\subsection{Some notation}\label{sec:uniqueness-notation}

Let $\iota^\alpha$ be the Brownian interlacement point process at level $\alpha$. 

Let $K>0$ be fixed. We decompose $\iota^\alpha$ into the point process $\iota'$ of trajectories that visit the ball $B(K)$ and $\iota''$ of trajectories that do not visit $B(K)$. Note that $\iota'$ and $\iota''$ are independent.

Let $N_K$ be the number of trajectories in $\iota'$. On the event $\{N_K=n\}$, 
\begin{itemize}
\item
let $(X_i,X_i')$, $1\leq i\leq n$, be the locations of the first and last visits to $B(K)$ of the interlacement trajectories from $\iota'$; and
\item
let $(\gamma_i, \widetilde \gamma_i, \gamma_i')$, $1\leq i\leq n$, be the three fragments of the interlacement trajectories $\iota'$, respectively, the time reversal of the part before the first entrance in $B(K)$, the part between the first entrance and the last visit in $B(K)$, and the part after the last visit in $B(K)$. 
\end{itemize}
Recall from Section~\ref{sec:BIPP}, that given $N_K=n$, $(X_i,X_i')$, $1\leq i\leq n$, are i.i.d.\ with distribution \eqref{eq:sample-1}, and given their locations on $\partial B(K)$, $(\gamma_i, \widetilde \gamma_i, \gamma_i')$, $1\leq i\leq n$, are independent with law \eqref{eq:sample-2}. 
We define 
\[
\mathcal I_{K,n} = \bigcup\limits_{w^*\in \iota''}B(w^*,1)\cup\bigcup\limits_{i=1}^n B(\gamma_i,1)\cup\bigcup\limits_{i=1}^n B(\gamma_i',1)
\]
and $\mathcal V_{K,n} = \R^d\setminus \mathcal I_{K,n}$. 
(Note that the Brownian interlacement $\mathcal I^\alpha$ is obtained from $\mathcal I_{K,n}$ by adding to it the $1$-sausages $B(\widetilde\gamma_i,1)$ around the bridges $\widetilde \gamma_i$, i.e.\ $\mathcal I^\alpha = \mathcal I_{K,n}\cup\bigcup\limits_{i=1}^n B(\widetilde \gamma_i,1)$.)

\medskip

The following connectivity property of $\mathcal I_{K,n}$ is crucial for the ``local surgery'' argument in the proofs of Proposition~\ref{prop:N=k} and Lemma~\ref{l:trifurcation}. It follows directly from Corollary~\ref{cor:condBI-connectivity-int}.

For $A\subseteq \R^d$ and $r>0$, let 
\[
\mathrm{int}_r(A) = \{x\in A\,:\,B(x,r)\subseteq A\}
\]
be the $r$-interior of $A$. 
\begin{proposition}\label{prop:r-connectivity-IKn}
Given $N_K=n$, 
\begin{equation}\label{eq:r-connectivity-IKn}
\text{$X_i$ is connected to $X_i'$ in $\mathrm{int}_1(\mathcal I_{K,n})$ for all $1\leq i\leq n$ almost surely.}
\end{equation}
\end{proposition}
\begin{proof}
For $\omega = \sum\limits_{i\geq 1}\delta_{(w_i^*,\alpha_i)}$ and $0\leq \alpha'<\alpha''$, we denote by $\iota^{\alpha',\alpha''}$ the Brownian interlacement point process with labels between $\alpha'$ and $\alpha''$: 
\[
\iota^{\alpha',\alpha''}=
\iota^{\alpha',\alpha''}(\omega) = \sum\limits_{i\geq 1:\,\alpha_i\in[\alpha',\alpha'']}\delta_{w_i^*}.
\]
Note that $\iota^{\alpha',\alpha''}$ is independent from $\iota^{\alpha'} (= \iota^{0,\alpha'})$ and has the same distribution as $\iota^{\alpha'' - \alpha'}$.

\smallskip

For $\varepsilon>0$, we write the Brownian interlacement point process $\iota^\alpha$ as the sum of independent interlacement point processes $\iota^{\alpha-\varepsilon}+\iota^{\alpha-\varepsilon,\alpha}=:\hat\iota+\check\iota$. If $\varepsilon$ is small enough, then with high probability none of the trajectories from $\check\iota$ visits $B(K)$; so by Corollary~\ref{cor:condBI-connectivity-int} one can connect $X_i$ and $X_i'$ inside the 
$1$-interior of the union of $1$-sausages around the conditional Brownian motions $\gamma_i$, $\gamma_i'$, and the trajectories of $\check\iota$, which is a subset of $\mathrm{int}_1(\mathcal I_{K,n})$. 
Since it works for any $\varepsilon$, the result follows. 

More precisely, let $\hat N_K$ be the number of trajectories of $\hat \iota$ that visit $B(K)$; and given $\hat N_K = \hat n$, define $(\hat X_i,\hat X_i')$ and $(\hat\gamma_i,\hat\gamma_i')$ for $\hat \iota$ analogously to $(X_i,X_i')$ and $(\gamma_i,\gamma_i')$ for $\iota^\alpha$ above. 

Since $\hat\iota$ and $\check\iota$ are independent, by Corollary~\ref{cor:condBI-connectivity-int} almost surely for every $1\leq i\leq \hat n$, there exist $\check w^*_{i,1},\ldots,\check w^*_{i,k_i}$ in the support of $\check \iota$, so that $\hat\gamma_i\stackrel{1}\longleftrightarrow \check w^*_{i,1}, \check w^*_{i,1}\stackrel{1}\longleftrightarrow \check w^*_{i,2},\ldots, \check w^*_{i,k_i}\stackrel{1}\longleftrightarrow \hat\gamma_i$. 
We denote this event by $\hat A$ and the event in \eqref{eq:r-connectivity-IKn} by $A$. 
Note that $\hat A$ implies $A$ if none of the trajectories of $\check\iota$ visits $B(K)$. Denote the latter event by $G$ and note that $\P[G]\xrightarrow[\varepsilon\to0]{} 1$. Then
\[
\P[A,\,N_K=n] \geq \P[\hat A\cap G,\,N_K=n] = \P[G,\,N_K = n]\xrightarrow[\varepsilon\to0]{}\P[N_K = n],
\]
so $\P[A\,|\,N_K=n] = 1$. 
\end{proof}

\subsection{Number of infinite components is \texorpdfstring{$0$}{0}, \texorpdfstring{$1$}{1} or \texorpdfstring{$\infty$}{infinite}}

In this section we rule out the case $N=k$ a.s.\ for some $2\leq k<\infty$. 
We begin with an auxiliary lemma. 
\begin{lemma}\label{l:N=k-complement}
Let $N=k$ a.s.\ for some $2\leq k<\infty$. Then for every $R>0$, $\mathcal V^\alpha\setminus B(R)$ contains exactly $k$ infinite components a.s.
\end{lemma}
\begin{proof}
Since $N=k$ a.s., $\mathcal V^\alpha\setminus B(R)$ contains at least $k$ infinite components a.s. 

\smallskip

Let us assume that for some $K$, $\mathcal V^\alpha\setminus B(K)$ contains at least $k+1$ infinite components with positive probability and denote the respective event by $E_K$. Note that $\P[E_{K'}]>0$ for all $K'>K$, since $E_K\subseteq E_{K'}$. 

Fix $n$, such that the probability of event $E_{K,n}=E_K\cap\{N_K=n\}$ is positive. By choosing $K$ large enough, we may assume that $n\geq 1$.

Since the bridges $\widetilde \gamma_i$ from $X_i$ to $X_i'$ have finite range, if $E_{K,n}$ occurs, then for large enough $R$, $\mathcal V_{K,n}\setminus B(R) = \mathcal V^\alpha\setminus B(R)$, in particular, $\mathcal V_{K,n}\setminus B(R)$ contains at least $k+1$ infinite connected components. Thus, if we denote for $R>0$ and $\delta>0$ by $E_{K,n,R,\delta}'$ the event that 
\begin{itemize}
\item[(a)]
$N_K=n$,  
\item[(b)]
$\mathcal V_{K,n}\setminus B(R)$ contains at least $k+1$ infinite connected components, 
\item[(c)]
$\|X_i-X_i'\|\geq \delta$ for all $1\leq i\leq n$, 
\end{itemize}
then $\P[E_{K,n,R,\delta}']>0$ for large enough $R$ and small enough $\delta$. 

\smallskip

Let $G$ be the event that every $1$-sausage $B(\widetilde\gamma_i,1)$ around the Brownian bridge from $X_i$ to $X_i'$ covers $B(R)$. Then, $\P\big[G\,|\,N_K=n,\,\|X_i-X_i'\|\geq \delta,\text{ for all }i\big]>c>0$. Hence $\P[E_{K,n,R,\delta}'\cap G]>0$. 
Now, if $E_{K,n,R,\delta}'\cap G$ occurs, then $\mathcal V^\alpha = \mathcal V^\alpha\setminus B(R)$ contains at least $k+1$ infinite connected components with positive probability, which contradicts the assumption that $N=k$ a.s. The proof is completed. 
\end{proof}

\begin{proposition}\label{prop:N=k}
$\P[N=k] = 0$ for all $2\leq k<\infty$.
\end{proposition}
\begin{proof}
Assume on the contrary that $N=k$ a.s.\ for some $2\leq k<\infty$. 

Fix $K$ large enough, so that the ball $B(K-2)$ intersects all $k$ infinite connected components with positive probability. Denote this event by $E_K$.

Fix $n$ such that the probability of event $E_{K,n}=E_K\cap\{N_K=n\}$ is positive. 
On the event $E_{K,n}$, the infinite connected component of $\mathcal V_{K,n}$ is unique and contains $B(K-2)$. 

\smallskip

Next, we explain that the Brownian bridges $\widetilde \gamma_i$ between $X_i$ and $X_i'$ can be rerouted with positive probability, so that the respective $1$-sausages do not cut the infinite component of $\mathcal V_{K,n}$ into several infinite components. 
Here, we use Proposition~\ref{prop:r-connectivity-IKn} to find paths from $X_i$ to $X_i'$ in the $1$-interior of $\mathcal I_{K,n}$; we then reroute the Brownian bridges from $X_i$ to $X_i'$ through thin tunnels around the respective paths. The corresponding $1$-sausages can only intersect $\mathcal V_{K,n}$ close to its boundary; thus, if the tunnels are thin enough, the sausages cannot cut the unique infinite component of $\mathcal V_{K,n}$ into several infinite components. 
This implies that $\mathcal V^\alpha$ contains a unique infinite component with positive probability and leads to a contradiction with the initial assumption. 

\smallskip

We now present the details of the construction. Since the bridges $\widetilde\gamma_i$ have finite range and by Lemma~\ref{l:N=k-complement}, if the event $E_{K,n}$ occurs, then $\mathcal V_{K,n}\setminus B(R)$ contains exactly $k$ infinite connected components for all large enough $R$ a.s. Let $E_{K,n,R}'$ be the event that 
\begin{itemize}
\item[(a)]
$N_K = n$,
\item[(b)]
the unique infinite connected component $\mathcal C'$ of $\mathcal V_{K,n}$ contains $B(K-2)$ and $\mathcal C'\setminus B(R)$ contains exactly $k$ infinite connected components,
\item[(c)]
for all $1\leq i\leq n$, $X_i$ is connected to $X_i'$ in $\mathrm{int}_1(\mathcal I_{K,n})\cap B(R-2)$.
\end{itemize}
By Proposition~\ref{prop:r-connectivity-IKn}, $\P[E_{K,n,R}']>0$ for $R$ large enough. 

On the event $E_{K,n,R}'$ there exist $n$ paths  in $\mathrm{int}_1(\mathcal I_{K,n})\cap B(R-2)$ from $X_i$ to $X_i'$ as well as $k$ paths in $\mathcal V_{K,n}$ connecting each of the $k$ infinite connected components of $\mathcal V_{K,n}\setminus B(R)$ to $B(K-2)$. 
Each of the $k$ vacant paths is contained in $\mathcal V_{K,n}$ together with some open neighborhood; therefore, there exist $\delta$ small enough and nearest neighbor paths 
$\rho_j = (y_{j,0},\dots, y_{j,m_j})$, $1\leq j\leq k$, on the lattice $\delta\Z^d$,
such that 
\begin{itemize}
\item
for all $j$, $y_{j,0}\in B(K-2)$, $y_{j,m_j}\notin B(R)$,
\end{itemize}
and with positive probability, 
\begin{itemize}
\item
for all $j$, $B(\rho_j, 5\sqrt{d}\delta) = \bigcup\limits_{m=0}^{m_j} B(y_{j,m},5\sqrt{d}\delta)\subseteq \mathcal V_{K,n}$;
\item
$y_{1,m_1},\ldots, y_{k,m_k}$ belong to different infinite connected components of $\mathcal V_{K,n}\setminus B(R)$.
\end{itemize}
Furthermore, there exist nearest neighbor paths 
$\pi_i = (x_{i,0},\ldots, x_{i,n_i})$, $1\leq i\leq n$, on the lattice $\delta\Z^d$, 
such that with positive probability, 
\begin{itemize}
\item
for all $i$, $X_i\in B(x_{i,0},\sqrt{d}\delta)$, $X_i'\in B(x_{i,n_i},\sqrt{d}\delta)$;
\item
for all $i$, $\pi_i$ is contained in a $\sqrt{d}\delta$-neighborhood of a path in $\mathrm{int}_1(\mathcal I_{K,n})\cap B(R-2)$ from $X_i$ to $X_i'$.
\end{itemize}
We denote this subevent of $E_{K,n,R}'$ by $F = F\big(K,n,R,\delta, (\rho_j)_j, (\pi_i)_i\big)$. 

\smallskip

For $1\leq i\leq n$, let $T_i = B(\pi_i, 2\sqrt{d}\delta) = \bigcup\limits_{m=0}^{n_i} B(x_{i,m},2\sqrt{d}\delta)$. 
Note that $T_i$ is contained in a $3\sqrt{d}\delta$-neighborhood of $\mathrm{int}_1(\mathcal I_{K,n})\cap B(R-2)$. 

Conditioned on $X_i\in B(x_{i,0},\sqrt{d}\delta)$ and $X_i'\in B(x_{i,n_i},\sqrt{d}\delta)$, the probability that the Brownian bridge $\widetilde \gamma_i$ from $X_i$ to $X_i'$ stays in $T_i$ is $\geq c_i>0$. 

Denote by $G$ the event that for all $i$, $X_i\in B(x_{i,0},\sqrt{d}\delta)$, $X_i'\in B(x_{i,n_i},\sqrt{d}\delta)$, and $\widetilde \gamma_i\subset T_i$. Then 
\[
\P[F\cap G] = \mathbb E\big[\mathds{1}_F\,\P[G\,|\,N_K=n,\,(X_i,X_i')_{1\leq i\leq n},\, \mathcal I_{K,n}]\big]\geq \P[F]\,\prod\limits_{i=1}^nc_i>0.
\]

It remains to notice, that when $F\cap G$ occurs, then the $1$-sausages $B(\widetilde\gamma_i,1)$ around the bridges $\widetilde \gamma_i$ (a) do not disconnect the points $y_{j,m}$, $1\leq j\leq k$, $0\leq m\leq m_j$, in the vacant set $\mathcal V_{K,n}$ and (b) do not intersect $B(K-2)\cup B(R)^c$. Thus, $F\cap G$ implies the uniqueness of the infinite connected component in the vacant set $\mathcal V^\alpha$; that is $\P[N=1]>0$ and we have obtained a contradiction with the assumption that $\P[N=k]=1$ for some $2\leq k<\infty$. The proof is completed. 
\end{proof}

\subsection{Expected number of large components}\label{sec:expected-number}

In this section we prove that the expected number of big components of $\mathcal V^\alpha$ in balls is finite; this, in particular, implies that any finite set is intersected by only finitely many infinite components of $\mathcal V^\alpha$ in expectation. 
We use this fact in the next section to rule out the possibility of infinitely many infinite connected components in $\mathcal V^\alpha$. 

\smallskip

For $r<R$, let $\mathrm{An}(r,R)=\{x\in\R^d\,:\,r\leq \|x\|\leq R\}$.

\begin{theorem}\label{thm:expected-number-components}
Let $r<R$. Let $N^\alpha(r,R)$ be the number of connected components of $\mathcal V^\alpha\cap\mathrm{An}(r,R)$, which intersect both $\partial B(r)$ and $\partial B(R)$. Then 
\[
\E\big[N^\alpha(r,R)\big]<\infty,
\]
for all $\alpha>0$, $r>0$ and $R>r+7(d+1)$.
\end{theorem}
\begin{proof}
It suffices to prove that for all large $n$, $\P\big[N^\alpha(r,R)> n\big]\leq \frac{1}{n^2}$.

\smallskip

Let $n\geq 1$. For $1\leq k\leq d$, let $R_k = r + 7k$ and consider $x_{k,1},\ldots, x_{k,n}\in\partial B(R_k)$ such that 
\[
\partial B(R_k)\subseteq \bigcup\limits_{i=1}^n B(x_{k,i},Cn^{-\tfrac{1}{d-1}}),
\]
for some $C=C(d,R)$. By the pigeon hole principle, if $N^\alpha(r,R)> n$, then for each $1\leq k\leq d$, there is a ball $B(x_{k,i_k},Cn^{-\tfrac{1}{d-1}})$, which intersects at least two (large) vacant components. Thus, 
\begin{eqnarray*}
\P\big[N^\alpha(r,R)> n\big]
&\leq 
&\P\Big[\bigcap_{k=1}^d\bigcup\limits_{i_k=1}^n\Big\{\begin{array}{c}\text{$\mathcal V^\alpha\cap B(x_{k,i_k},Cn^{-\tfrac{1}{d-1}})$ contains at}\\ \text{least $2$ connected components}\end{array}\Big\}\Big]\\
&\leq 
&n^d\sup\limits_{\substack{x_k\in\partial B(R_k)\\ 1\leq k\leq d}}
\P\Big[\begin{array}{c}\text{$\mathcal V^\alpha\cap B(x_k,Cn^{-\tfrac{1}{d-1}})$ contains at least $2$}\\ \text{connected components for all $1\leq k\leq d$}\end{array}\Big].
\end{eqnarray*}
By the local picture of the Brownian interlacements (see Introduction), Theorem~\ref{thm:uniqueness-severalballs} immediately gives that 
for all $x_1,\ldots, x_J\in\R^d$ with $\|x_j-x_{j'}\|>6$ ($j\neq j'$), for all $\varepsilon\in(0,1)$ and $\alpha>0$, 
\begin{equation}\label{eq:uniqueness-severalballs-BI}
\P\Big[\begin{array}{c}\text{$\mathcal V^\alpha\cap B(x_j,\varepsilon)$ contains at least $2$ connected}\\ \text{components for all $1\leq j\leq J$}\end{array}\Big]
\leq \Big(C\log^m\big(\tfrac1\varepsilon\big)\varepsilon^{d+1}\Big)^J.
\end{equation}

\medskip

We apply \eqref{eq:uniqueness-severalballs-BI} with $\varepsilon = Cn^{-\tfrac{1}{d-1}}$ and $J=d$ to get for some $C=C(d,R)$ and $m=m(d,R)$ that 
\[
\P\big[N^\alpha(r,R)> n\big]\leq n^d\,\Big(C(\log^m n) n^{-\frac{d+1}{d-1}}\Big)^d.
\]
Since 
$n^d\,n^{-\frac{d+1}{d-1}d}=n^{-2 - \frac{2}{d-1}}$, the right hand side is smaller than $\tfrac{1}{n^2}$ for all large $n$. 

The proof is completed.
\end{proof}

\subsection{Ruling out infinitely many infinite components}

We follow the classical Burton-Keane argument, but with a slightly more general notion of trifurcation. 
\begin{definition}
Let $t>0$. We say that $x\in\R^d$ is a \emph{$t$-trifurcation} if there is an infinite connected component $\mathcal C$ of the vacant set $\mathcal V^\alpha$, such that 
\begin{itemize}
\item[(a)]
$x\in \mathcal C$;
\item[(b)]
$\mathcal C\setminus \mathcal C_{x,t}$ contains at least $3$ infinite connected components, where $\mathcal C_{x,t}$ is the connected component of $x$ in $\mathcal C\cap B(x,t)$. 
\end{itemize}
\end{definition}
\begin{lemma}\label{l:trifurcation}
Assume that $N=\infty$ a.s.\ Then there exist $t>0$, such that 
\[
\P\big[\text{$0$ is a $t$-trifurcation}\big]>0.
\]
\end{lemma}
\begin{proof}
The proof is very similar to the proof of Proposition~\ref{prop:N=k}. 

\smallskip

Fix a radius $K$ large enough, so that the ball $B(K-2)$ intersects at least $3$ infinite connected components with positive probability. Denote this event by $E_K$.

Fix $n$ such that the probability of event $E_{K,n}=E_K\cap\{N_K=n\}$ is positive. 
Since the Brownian bridges $\widetilde\gamma_i$ have finite range, if $E_{K,n}$ occurs, then $\mathcal V_{K,n}\setminus B(R) = \mathcal V^\alpha\setminus B(R)$ for all large $R$, in particular, $\mathcal V_{K,n}$ contains an infinite component $\mathcal C'$, such that $B(K-2)\subset \mathcal C'$ and $\mathcal C'\setminus B(R)$ contains at least $3$ infinite connected components.

For $R>0$, let $E_{K,n,R}'$ be the event that 
\begin{itemize}
\item[(a)]
$N_K = n$;
\item[(b)]
there is an infinite connected component $\mathcal C'$ in $\mathcal V_{K,n}$, such that $B(K-2)\subset\mathcal C'$ and $\mathcal C'\setminus B(R)$ contains at least $3$ infinite connected components;
\item[(c)]
for all $1\leq i\leq n$, $X_i$ is connected to $X_i'$ in $\mathrm{int}_r(\mathcal I_{K,n})\cap B(R-2)$.
\end{itemize}
By Proposition~\ref{prop:r-connectivity-IKn}, $\P[E_{K,n,R}']>0$ for all $R$ large enough. 

\smallskip

On the event $E_{K,n,R}'$ there exist $n$ paths  in $\mathrm{int}_1(\mathcal I_{K,n})\cap B(R-2)$ from $X_i$ to $X_i'$ as well as $3$ paths in $\mathcal V_{K,n}$ from $B(K-2)$ to different infinite connected components of $\mathcal V_{K,n}\setminus B(R)$. Now, precisely as in the proof of Proposition~\ref{prop:N=k}, one shows that with positive probability, the $1$-sausages $B(\widetilde \gamma_i,1)$ around the Brownian bridges $\widetilde \gamma_i$ from $X_i$ to $X_i'$ do not visit $B(K-2)\cup B(R)^c$ and do not disconnect the $3$ paths in $\mathcal V_{K,n}$ from $B(K-2)$ to $3$ different infinite connected components of $\mathcal V_{K,n}\setminus B(R)$. Thus, for $R$ large enough, with positive probability, $\mathcal V^\alpha$ contains an infinite connected component $\mathcal C$, such that $B(K-2)\subset\mathcal C$ and $\mathcal C\setminus B(R)$ contains at least $3$ infinite connected components. 
Call this event $F_{K,n,R}$.
We claim that $F_{K,n,R}$ implies that $0$ is a $t$-trifurcation for all $t$ large enough.

\smallskip

Assume that $F_{K,n,R}$ occurs and let $\{\mathcal C_i\}_{i\in I}$ be all $(\geq 3)$ the infinite connected components of $\mathcal C\setminus B(R)$, $\mathcal C_t$ the connected component of $0$ in $\mathcal C\cap B(t)$ and $\widetilde {\mathcal C}_t = \mathcal C\cap B(t)$.

We first claim that any $x\in \widetilde {\mathcal C}_R$ is connected to $0$ in $\mathcal C_t$ for all large $t$. Indeed, if it was the case that for every $t>R$, a $x_t\in \widetilde {\mathcal C}_R$ was not connected to $0$ in $\mathcal C_t$, then the number of vacant components in $\widetilde{\mathcal C}_{R'}$ that intersect both $\partial B(R)$ and $\partial B(R')$ would be infinite for any $R'>R$, which contradicts Theorem~\ref{thm:expected-number-components}.

Thus, let $t$ be large enough, so that $\widetilde{\mathcal C}_R$ is a subset of $\mathcal C_t$. We claim that $\mathcal C\setminus \mathcal C_t$ contains at least $3$ infinite connected components. Indeed, let $\mathcal C_i'$ be an infinite connected component of $\mathcal C_i\setminus B(t)$. (Note that $\mathcal C_i'\subseteq \mathcal C\setminus\mathcal C_t$.)
Since any $\mathcal C_i$ and $\mathcal C_j$ are not connected in $\mathcal C\setminus B(R)$, any $\mathcal C_i'$ and $\mathcal C_j'$ are not connected in $\mathcal C\setminus B(R) = \mathcal C\setminus \widetilde{\mathcal C}_R$. 
Since $\widetilde{\mathcal C}_R\subseteq\mathcal C_t$, any $\mathcal C_i'$ and $\mathcal C_j'$ are not connected in $\mathcal C\setminus\mathcal C_t$. Thus, $\mathcal C\setminus\mathcal C_t$ contains at least $3$ infinite connected components;  
by definition, $0$ is a $t$-trifurcation. 

Since $\P[F_{K,n,R}]>0$, $0$ is a $t$-trifurcation with positive probability for all $t$ large enough. The proof is completed. 
\end{proof}

\begin{proposition}\label{prop:N=01}
$\P[N=\infty] = 0$.
\end{proposition}
\begin{proof}
Assume on the contrary that $N=\infty$ a.s.
Let $t>0$ be as in Lemma~\ref{l:trifurcation}. 

By arguing exactly as in the proof of \cite[Theorem~2.4]{HJ-Uniqueness}, we notice that for any finite set of $t$-trifurcations $\mathcal T$ of an infinite connected component $\mathcal C$, such that $\mathcal C_{x,t}\cap\mathcal C_{x',t}=\emptyset$ for all different $x,x'\in\mathcal T$, the set $\mathcal C\setminus\bigcup\limits_{x\in\mathcal T}\mathcal C_{x,t}$ contains at least $|\mathcal T|+2$ infinite connected components. 
Thus, if $\mathcal C$ is an infinite cluster with $j$ $t$-trifurcations in a finite set $W$, such that the respective $j$ sets $\mathcal C_{x,t}$ are pairwise disjoint and do not intersect the boundary of $W$, then $\mathcal C\setminus W$ contains at least $j+2$ infinite connected components.

\smallskip

Let $L>2t$. For $n>0$, let $\mathcal T_n$ be the set of all $t$-trifurcations in $(L\Z^d)\cap (-nL,nL)^d$ and let $\mathcal Y_n$ be the set of 
infinite connected components of $\mathcal V^\alpha\setminus (-nL,nL)^d$ which intersect $\partial[-nL,nL]^d$. From the above observation, we obtain that 
$|\mathcal T_n|\leq |\mathcal Y_n| - 2$.
In particular, $\E\big[|\mathcal T_n|\big]\leq \E\big[|\mathcal Y_n|\big]$. 

\smallskip

On the one hand, by Lemma~\ref{l:trifurcation}, $\E\big[|\mathcal T_n|\big]\geq c (2n-1)^d$ for some $c>0$. 

On the other hand, there exist $z_1,\ldots, z_m\in\R^d$ with $m\leq Cn^{d-1}$ for some $C$, such that 
$\partial [-nL,nL]^d\subset \bigcup\limits_{i=1}^m B(z_i,1)$. Let $N_i$ be the number of infinite connected components in $\mathcal V^\alpha\setminus \mathrm{int}(B(z_i,1))$ which intersect $\partial B(z_i,1)$. Since every infinite connected component from $\mathcal Y_n$ intersects at least one of the balls $B(z_i,1)$, we have that $|\mathcal Y_n|\leq \sum\limits_{i=1}^m N_i$. 
Therefore, $\E\big[|\mathcal Y_n|\big]\leq \sum\limits_{i=1}^m \E[N_i] = m\E[N_0]\leq Cn^{d-1}\E[N_0]$, where $N_0$ is the number of infinite connected components in $\mathcal V^\alpha\setminus \mathrm{int}(B(1))$ which intersect $\partial B(1)$. 

\smallskip

As a result, $\E[N_0]\geq \frac{c\,n}{C}$ for all $n$. Hence $\E[N_0]=\infty$, which contradicts with Theorem~\ref{thm:expected-number-components}. The proof is completed. 
\end{proof}

\end{document}